\documentclass[12pt, twoside, leqno]{article}

\usepackage{amsmath,amsthm}
\usepackage{amssymb}

\usepackage{enumerate}

\usepackage{graphicx}

\usepackage[T1]{fontenc}

\numberwithin{equation}{section}

\usepackage[utf8]{inputenc}
\usepackage{enumitem}
\usepackage{cite}
\usepackage[capitalize,nameinlink]{cleveref}

\makeatletter
\newif\ifHy@pdfstring
\def\texorpdfstring{%
\ifHy@pdfstring
\expandafter\@secondoftwo
\else
\expandafter\@firstoftwo
\fi
}
\makeatother

\crefname{empty}{}{}
\newlist{alist}{enumerate}{1}
\setlist[alist]{label=(\alph*),itemsep=.1ex, ref=(\alph*)}
\crefalias{alisti}{empty}
\newlist{rlist}{enumerate}{1}
\setlist[rlist]{label=(\roman*),itemsep=.1ex, ref=(\roman*)}
\crefalias{rlisti}{empty}
\newlist{nlist}{enumerate}{1}
\setlist[nlist]{label=(\arabic*),itemsep=.1ex, ref=(\arabic*)}
\crefalias{nlisti}{empty}
\crefname{equation}{}{}
\Crefname{figure}{Figure}{Figures}
\crefname{page}{page}{pages}
\Crefname{enumi}{}{}
\Crefname{subsection}{Subsection}{Subsections}
\def\theoremname{Theorem}%
\def\propositionname{Proposition}%
\def\lemmaname{Lemma}%
\def\corollaryname{Corollary}%
\def\definitionname{Definition}%
\def\conventionname{Convention}%
\def\axiomname{Axiom}%
\def\remarkname{Remark}%
\def\examplename{Example}%
\def\questionname{Question}%
\def\constructionname{Construction}%
\def\notationname{Notation}%
\def\assumptionname{Assumption}%
\def\conjecturename{Conjecture}%
\newtheorem{thm}{\theoremname}[section]
\theoremstyle{plain}
\newtheorem{theorem}[thm]{\theoremname}
\newtheorem{proposition}[thm]{\propositionname}
\newtheorem{lemma}[thm]{\lemmaname}
\newtheorem{corollary}[thm]{\corollaryname}
\theoremstyle{definition}
\newtheorem{definition}[thm]{\definitionname}

\newtheorem{remark}[thm]{\remarkname}

\newtheorem{notation}[thm]{\notationname}

\newtheorem{observation}[thm]{Observation}

\providecommand\ordinarycolon{:}
\def\vcentcolon{\mathrel{\mathop\ordinarycolon}}
\newcommand{\coloneqq}{\vcentcolon\mathrel{\mkern-1.2mu}=}
\newcommand{\eqqcolon}{=\mathrel{\mkern-1.2mu}\vcentcolon}
\newcommand{\formatDefinition}[1]{\emph{#1}}
\newcommand{\abs}[1]{\left\lvert#1\right\rvert}
\DeclareMathOperator{\interiorop}{int}
\newcommand{\interior}[1]{\interiorop(#1)}
\newcommand{\norm}[1]{\left\lVert#1\right\rVert}
\newcommand{\field}[1]{\mathbb{#1}}
\DeclareMathOperator{\lip}{lip}
\DeclareMathOperator{\Lip}{Lip}
\newcount\Level
\let\Part=\part\def\part{\global\Level=0\Part}
\let\Chapter=\chapter\def\chapter{\global\Level=1\Chapter}
\let\Section=\section\def\section{\global\Level=2\Section}
\let\Subsection=\subsection\def\subsection{\global\Level=3\Subsection}
\let\Subsubsection=\subsubsection\def\subsubsection{\global\Level=4\Subsubsection}

\def\levelText{%
\ifcase\the\Level part %
\or chapter %
\or section %
\or subsection %
\or subsubsection %
\else part \fi}

\newcommand{\clopen}[2]{\char91 #1,#2\char41}
\newcommand{\oped}[2]{\char40 #1,#2\char93}
\newcommand{\convFromRight}{{\scriptscriptstyle \searrow}}
\newcommand{\convFromRightB}{{\scriptstyle \searrow}\,}

\begin{document}


\baselineskip=17pt


\title{On sets where $\lip f$ is finite}

\author{Zoltán Buczolich\\
ELTE E\"otv\"os Lor\'and University\\
Pázmány Péter Sétány 1/c\\
1117 Budapest, Hungary\\
www.cs.elte.hu/\hbox{$\sim$}buczo\\
ORCID ID:0000-0001-5481-8797\\
E-mail: buczo@cs.elte.hu
\and
Bruce Hanson\\
Department of Mathematics\\
Statistics and Computer Science\\
St.~Olaf College\\
Northfield\\
Minnesota 55057, USA\\
E-mail: hansonb@stolaf.edu
\and
Martin Rmoutil\\
Department of Mathematics Education\\
Faculty of Mathematics and Physics\\
Charles University\\
Sokolovská 83\\
186 75 Prague 8, Czech Republic\\
E-mail: martin@rmoutil.eu
\and
Thomas Zürcher\\
Instytut Matematyki\\
Uniwersytet Śląski\\
Bankowa 14\\
40-007 Katowice, Poland\\
E-mail: thomas.zurcher@us.edu.pl
}

\date{}

\maketitle


\renewcommand{\thefootnote}{}



\renewcommand{\thefootnote}{\arabic{footnote}}
\setcounter{footnote}{0}


\begin{abstract}
Given a function $f\colon \field{R}\to \field{R}$, the so-called \lq\lq little lip\rq\rq\ function $\lip f$ is defined as follows:
\begin{equation*}
\lip f(x)=\liminf_{r\convFromRight 0}\sup_{\abs{x-y}\le r} \frac{\abs{f(y)-f(x)}}{r}.
\end{equation*}
We show that if $f$ is continuous on $\field{R}$, then the set where $\lip f$ is infinite is a countable union of a countable intersection of closed sets (that is an $F_{\sigma \delta}$ set). On the other hand, given a countable union of closed sets $E$, we construct a continuous function $f$ such that $\lip f$ is infinite exactly on $E$. A further result is that for the typical continuous function $f$ on the real line $\lip f$ vanishes almost everywhere.
\end{abstract}
\section{Introduction}
Throughout this section we will assume that $f$ is a continuous real-valued function that is defined on $\field{R}$. The so-called \lq\lq big Lip\rq\rq\ function, $\Lip f$, is defined as follows:
\begin{equation*}
  \Lip f(x)=\limsup_{y\to x} \frac{\abs{f(y)-f(x)}}{\abs{y-x}}=\limsup_{r\convFromRight 0} \sup_{\abs{x-y}\le r}\frac{\abs{f(y)-f(x)}}r.
\end{equation*}
According to the Rademacher--Stepanov Theorem, \cite{Stepanov}, $f$ is differentiable almost everywhere on the set
\begin{equation*}
L_f=\{x \in \field{R} : \Lip f(x) < \infty\}.
\end{equation*}
More recently, the \lq\lq little lip\rq\rq\ function, $\lip f$, which is defined as follows, has been investigated:
\begin{equation*}
  \lip f(x)=\liminf_{r\convFromRight 0} \sup_{\abs{x-y}\le r}\frac{\abs{f(y)-f(x)}}r.
\end{equation*}
For example, $\lip f$ shows up in J.~Cheeger's seminal paper \cite{Cheeger}, in which he shows that in quite general metric measure spaces a version of Rademacher's theorem holds. It also features prominently in \cite{Keith}, where it makes its appearance as part of a sufficient condition for a version of Rademacher's theorem.
In \cite{BaloghCsornyei} Balogh and Csörnyei show that the Rademacher--Stepanov Theorem does not remain true if we replace $L_f$ with
\begin{equation*}
l_f=\{x \in \field{R} : \lip f(x) < \infty\}.
\end{equation*}
In fact, they produce an example where $\lip f(x) =0$ almost everywhere, but $f$ is nowhere differentiable.  In \cite{Hanson} Hanson, the second author of this paper, sharpens their result to show that the exceptional set where $\lip f(x) \neq 0$ can even be made to have Hausdorff dimension 0.   On the other hand, Balogh and Csörnyei also show that if $\field{R}\backslash l_f$ is countable, then  every interval contains a set of positive measure on which $f$ is differentiable.

Given the relationship between $L_f$ and $l_f$ and the set of  differentiability of~$f$, it is interesting to determine the possible structure of the sets $L_f$ and $l_f$.  It is not difficult to show that a set $E\subset \field{R}$ is equal to $L_f$ for some continuous function~$f$ if and only if $E$ is an $F_\sigma$ set, that is $E$ is a countable union of closed sets (see \cref{Baire classes}~\cref{Capital Lf is Fsigma} and \cref{g delta theorem}).  On the other hand, characterizing the set $l_f$ is more difficult.  Recall that a $G_{\delta\sigma}$ set is a set that can be written as countable union of  countable intersections of open sets. It is straightforward to show that every $l_f$ is a $G_{\delta\sigma}$ set (see \cref{Baire classes}~\cref{lf is Gdeltasigma}), and we conjecture that the converse is also true, namely that every $G_{\delta\sigma}$ set is equal to $l_f$ for some continuous function $f$. Determining the truth of this conjecture appears to be quite difficult.   The main result of this paper is to show that
for every $G_\delta$ set $E$ there is a continuous function $f$ such that
$l_f=E$, and already the proof of this result requires a bit of work.

On a related note, according to a result of Banach \cite{Banach},
for the typical continuous function $f\in C([0,1])$, we have $\Lip f(x)=\infty$ for all $x\in [0,1]$.
In contrast to this result, we show that, somewhat surprisingly, for the typical continuous function $f$  we have  $\lip f$ equal to~$0$ at points of a residual set of full measure.

The structure of the paper is as follows:  In \cref{Notation and Basic Results} we introduce some notation and present a few basic results about $L_f$ and $l_f$.  \cref{The Main Result} contains the proof of the main result that for every $G_\delta$ set $E$ there exists a continuous function $f$ such that $l_f=E$.  We also show that if $E$ is an $F_\sigma$ set, then $E=l_f$ for some continuous function $f$.  Finally, in \cref{Typical Results} we show that the typical continuous function has $\lip$ equal to~$0$ at points of a residual set of full measure.

\section{Notation and Basic Results}\label{Notation and Basic Results}
We denote by $\overline{A}$ the closure of the set $A$.\\
A set $E$ in a complete metric space $X$ is of \formatDefinition{first Baire category}, also called \formatDefinition{meager}, if it is the
union of countably many nowhere dense sets.  We say that a property is \formatDefinition{typical}, also known as \formatDefinition{generic}, in  $X$ if the set of those $x\in X$ that do not have this property is meager.

After introducing some notation concerning the sets where $\lip$ and $\Lip$ are finite and infinite, respectively, we determine their places in the Borel hierarchy.
\begin{definition}[The sets $L_f$, $L_f^\infty$, $l_f$ and $l_f^\infty$]
For a continuous function $f\colon \field{R}\to \field{R}$, we set
\begin{align*}
L_f&=\{x\in \field{R}: \Lip f(x)<\infty\},\\
L_f^\infty&=\{x\in \field{R}: \Lip f(x)=\infty\},\\
l_f&=\{x\in \field{R}: \lip f(x)<\infty\},\\
l_f^\infty&=\{x\in \field{R}: \lip f(x)=\infty\}.
\end{align*}
\end{definition}

\subsection{Baire classes of \texorpdfstring{$l_f$}{l{}f} and \texorpdfstring{$L_f$}{L{}f}}

\begin{definition}[The functions $q_f(x,r)$, $l_f(x,r)$, and $L_f(x,r)$]
We assume that $f\colon I \to \field{R}$ is a function, where $I\subset \field{R}$ is a closed subinterval of $\field{R}$ or $\field{R}$ itself. We let $r$ and $R$ be positive numbers. Then we define the following quantities:
\begin{itemize}
\item $q_f(x,r) = \sup_{y\in [x-r,x+r]\cap I} \frac{|f(y)-f(x)|}{r}$,
\item $l_f(x,R) = \inf_{r\in(0,R)} q_f(x,r)$,
\item $L_f(x,R) = \sup_{r\in(0,R)} q_f(x,r)$.
\end{itemize}
\end{definition}

\begin{observation}\label{O:propertiesofqlL}
Let $f\colon \field{R}\to\field{R}$ be a continuous function. Then the following statements hold.
\begin{alist}
\item\label{qf continuous} Let $r>0$ be fixed. Then $q_f(\cdot, r)$ is a continuous function.
\item\label{upper and lower semicontinuity} With $R>0$ fixed, $l_f(\cdot,R)$ is upper semi-continuous and $L_f(\cdot,R)$ is lower semi-continuous.
\item\label{definition of lip} $\lip f(x) = \lim_{R\convFromRight 0} l_f(x,R) = \lim_{n\to \infty} l_f(x, 1/n) = \sup_{n\in\field{N}} l_f(x, 1/n)$.
\item\label{definition of big Lip} $\Lip f(x) = \lim_{R\convFromRight 0} L_f(x,R) = \lim_{n\to \infty} L_f(x, 1/n) = \inf_{n\in\field{N}} L_f(x, 1/n)$.
\end{alist}
\end{observation}

\begin{proof}
The proofs of
\cref{qf continuous,definition of lip,definition of big Lip}
  are trivial from the definitions. The proof of \cref{upper and lower semicontinuity} uses \cref{qf continuous} and the fact that upper semi-continuous functions are closed under taking infima, and lower semi-continuous functions are closed under taking suprema.
\end{proof}

\begin{lemma}\label{Baire classes}
Let $f\colon \field{R}\to \field{R}$ be a continuous function. In this case, the following statements are valid.
\begin{alist}
\item\label{lf is Gdeltasigma} The set $l_f$ is a $G_{\delta\sigma}$ set.
\item\label{Capital Lf is Fsigma} The set $L_f$ is an $F_\sigma$ set.
\item\label{lip zero is Gdelta} The set $\{x:\lip f(x)=0\}$ is a $G_\delta$ set.
\end{alist}
\end{lemma}

\begin{proof}
To prove \cref{lf is Gdeltasigma}, we note that:
\begin{equation*}
\begin{aligned}
\lip f(x)=\infty \quad \Longleftrightarrow & \quad \forall k\in\field{N} : \quad \lip f(x)> k  \\
 \Longleftrightarrow & \quad \forall k\in\field{N} \quad \exists n\in \field{N} : \quad l_f\biggl(x,\frac{1}{n}\biggr) \geq k \\
 \Longleftrightarrow & \quad \forall k\in\field{N} \quad \exists n\in \field{N} \quad \forall r\in\biggl(0,\frac{1}{n}\biggr) : \quad q_f(x,r)\geq k \\
 \Longleftrightarrow & \quad x\in \; \bigcap_{k\in\field{N}} \; \bigcup_{n\in\field{N}} \; \bigcap_{r\in (0,1/n)} \left \{ z\in\field{R}: q_f(z,r) \geq k \right\}.
\end{aligned}
\end{equation*}

It now follows from the continuity of $q_f$ that $l_f^\infty$ is an $F_{\sigma\delta}$ set, establishing~\cref{lf is Gdeltasigma}.  The proofs of \cref{Capital Lf is Fsigma,lip zero is Gdelta} are similar and left to the reader.
\end{proof}

\section{The Main Result}\label{The Main Result}
\begin{theorem}
\label{F sigma}
Let $F\subset \field{R}$ be an $F_\sigma$ set. Then there exists a continuous function $f\colon \field{R}\to \field{R}$ such that $l_f^\infty=F$.
\end{theorem}

The proof of \cref{F sigma} is rather involved and will be accomplished in a sequence of steps.  We start with supposing that $F$ is countable. In the second step, we handle a nowhere dense perfect set, which prepares us for the case where $F$ is the countable union of nowhere dense perfect sets. Having established this case, we next look at the case where the sets are nowhere dense and merely closed.  In the final step, we consider the general case: a countable union of closed sets.

\subsection{The countable case}

\begin{theorem}\label{countable theorem}
Given any countable set of points $S$ in $\field{R}$, there exists a continuous, increasing function $f\colon\field{R}\to \field{R}$ such that
\begin{equation}\label{l_f^infty=S}
l_f^\infty=S
\end{equation}
and
\begin{equation}\label{Lip f finite off of S}
\Lip f(x)< \infty \text{ for all }x \notin \overline{S}.
\end{equation}
\end{theorem}

If $S$ is a finite set, it is trivial to verify the \lcnamecref{countable theorem}, so we will assume throughout this section that $S$ is countably infinite, and we write henceforth $S=\{s_1,s_2,\ldots\}$. We also assume without loss of generality that $S \subset (0,1)$.

We begin with a few definitions.
\begin{definition}[acceptable for]\label{acceptable}
Given $x_0$ in $(0,1)$, we say that the sequence $\{x_n\}\subset (0,1)$ is \formatDefinition{acceptable for $x_0$} if
\begin{gather}\label{not in S}
x_n \notin S \mbox{ for } n=2,3,\ldots,\\
x_n \convFromRightB x_0\label{decreasing to x},\\
\intertext{and}
x_{n+1}-x_{n+2}<\frac13(x_n-x_{n+1}) \ \text{for all $n \in \field{N}$.}\label{geometric decay}
\end{gather}
\end{definition}

We note that given any $0<x_0<y<1$, we can easily choose a sequence $\{x_n\}$ that is acceptable for $x_0$ such that $x_1=y$.

\begin{definition}[The function $f_{\{x_n\}}$]\label{f_{x_n}}
Given a sequence $\{x_n\}$ that is acceptable for $x_0$, we define
$f=f_{\{x_n\}}$ on $[x_0,x_1]$ as follows:
\begin{gather}\label{acceptable_0}
f(x_0)=0,\\
f(x_n)=2^{-n+1}(x_1-x_0) \ \text{for all $n \in \field{N}$}\label{acceptable_1},\\
f \mbox{ is linear on each }I_n=[x_{n+1},x_n]\mbox{ for all $n\in\field{N}$}.\label{acceptable_2}
\end{gather}
\end{definition}

\begin{remark}\label{m_n goes to infinity}
Note that by \cref{acceptable_1} the slope~$m_n$ of $f$ on $I_n$ satisfies the equation $m_n=\frac{x_1-x_0}{2^n(x_n-x_{n+1})}$.

On the other hand, \cref{geometric decay} guarantees that\footnote{For $n=1$, the inequality follows from the fact that the sequence of the $x_n$ decreases to~$x_0$.} $x_n-x_{n+1}<\frac{x_1-x_0}{3^{n-1}}$, so we get $m_n>\frac{3^{n-1}}{2^n}\to \infty$, and therefore $\lip f(x_0)=\infty$.

\end{remark}

\begin{definition}[acceptable on]\label{g_acceptable}
Suppose that $g$ is defined on $[0,1]$ and $\{x_n\}$ is acceptable for $x_0 \in (0,1)$.  Then we say that $g$ is \formatDefinition{acceptable on $[x_0,x_1]$} if there is a constant $c$ such that $g=f_{\{x_n\}}+c$ on $[x_0,x_1]$.
\end{definition}
\begin{definition}[parallelogram $P_{L,\varepsilon}$ and $q_P$]\label{parallelogram}
Suppose that $0<\varepsilon<1$ and $L$ is a finite, closed line segment in $\field{R}^2$ with positive slope~$m$.  Let $A$ and $B$ denote the endpoints of $L$ and define $P=P_{L,\varepsilon}$ to be the closed parallelogram with $L$ as one of its diagonals and the boundary of $P$ made up of the line segments with $A$ and $B$ as endpoints and slopes $(1+\varepsilon)m$ and $(1-\varepsilon)m$.  We call $L$ the \formatDefinition{main diagonal} of $P$.  The situation is schematically represented in \cref{parallelogramFigure}.
We also define $q_P=(1+3\varepsilon)m$.
\end{definition}
\begin{figure}
\centering
\includegraphics{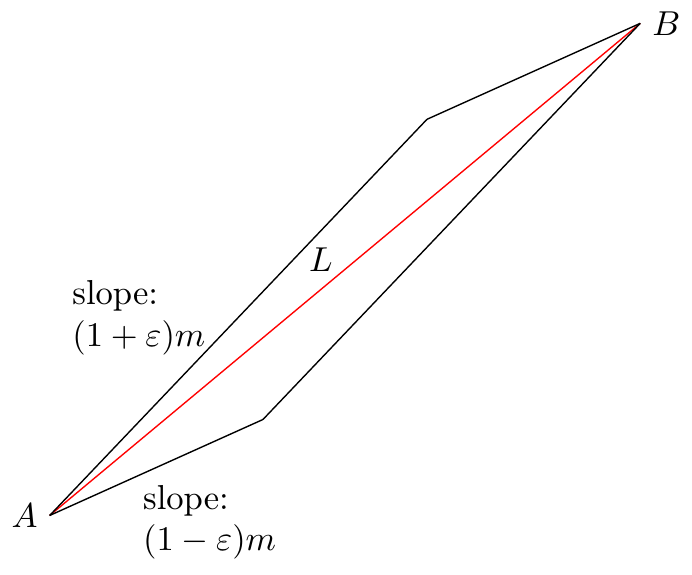}
\caption{Parallelogram as considered in \cref{parallelogram}}\label{parallelogramFigure}
\end{figure}

\begin{definition}[direct descendant]\label{direct descendant}
Suppose that $0<\delta, \varepsilon<1$, and $L$ is a finite, closed line segment in $\field{R}^2$ with positive slope,  moreover $P=P_{L,\varepsilon}$ and $Q=P_{M,\delta}$
are parallelograms with $M$ being a closed line segment sharing an endpoint with $L$ and $Q \subset P$.  Then we say that \formatDefinition{$Q$ is a direct descendant of $P$.}
\end{definition}

\begin{remark}\label{direct descendant inequality}
Note that if $Q$ is a direct descendant of $P$, then $\delta \leq \varepsilon$ and
$q_Q  \le q_P$.
\end{remark}

We next state a simple \lcnamecref{parallelogram estimate}, which will be useful in the proof of the existence of the function described in \cref{countable theorem}.  In the \lcnamecref{parallelogram estimate} we use the following notation:  Given two points $A,B \in \field{R}^2$, we define $[A,B]$ to be the closed line segment with $A$ and $B$ as the endpoints.  We leave the proof of the \lcnamecref{parallelogram estimate}, which is straightforward, to the reader.

\begin{lemma}\label{parallelogram estimate} Let $f\colon[a,b]\to \field{R}$ be continuous with $\frac{f(b)-f(a)}{b-a}=m>0$ and define the parallelogram $P$ as $P_{L,\varepsilon}$, where $L=[(a,f(a)),(b,f(b))]$ and $\varepsilon \le 1/2$.  Suppose that the graph of~$f$ is contained in~$P$.  Then for any $x \in (a,b)$ and $r=\min\{x-a,b-x\}$ we have
\begin{equation}\label{q estimate}
q_f(x,r)\leq q_P.
\end{equation}
\end{lemma}

\begin{definition}[line segment~$L_I$]\label{L_I}
If $f$ is linear on the interval $I=[a,b]$, then we define ${L_I=[(a,f(a)),(b,f(b))]}$.
\end{definition}

\begin{definition}[fundamental pair, fundamental envelope]\label{fundamental}
Suppose that $f$ is an increasing homeomorphism of $[0,1]$ onto itself and ${\mathcal I}=\{I_n\}$ is a countable collection of non-overlapping intervals that are closed and such that $\bigcup I_n =[0,1]$, $f$ is linear on each $I_n=[a_n,b_n]$ and each $a_n \notin S$.  Then we say that $(f,{\mathcal I})$ is a \formatDefinition{fundamental pair}.

  Finally, given a fundamental pair $(f,{\mathcal I})$ with ${\mathcal I}=\{I_n\}$ and $L_n=L_{I_n}$ and a sequence $\{\varepsilon_n\}$ where $0<\varepsilon_n\le 1/2$ for all $n \in \field{N}$,  we define
\begin{equation*}
P_{f,{\mathcal I},\{\varepsilon_n\}}=\bigcup_{n=1}^\infty P_{L_n,\varepsilon_n}
\end{equation*}
and call $P_{f,{\mathcal I},\{\varepsilon_n\}}$ a \formatDefinition{fundamental envelope}  of $(f,{\mathcal I})$.
Thus,  $P_{f,{\mathcal I},\{\varepsilon_n\}}$ is a union of non-overlapping, closed parallelograms, which cover the graph of $f$.  Note that we require that $\varepsilon_n \le 1/2$ for all $n \in \field{N}$.
\end{definition}

\begin{definition}[successor]\label{successor}
Suppose that $(f,{\mathcal I})$ is a fundamental pair with ${\mathcal I}=\{I_n\}=\{[a_n,b_n]\}$.  We say that the fundamental pair $(g,{\mathcal J})$ is a \formatDefinition{successor} to $(f,{\mathcal I})$ if there is $m \in \field{N}$ such that $g=f$ on $[0,a_m]\cup [b_m,1]$ and there are $\beta_1,\beta_2 \in (a_m,b_m)$ such that $[a_m,\beta_1],[\beta_2,b_m]\in {\mathcal J}$ and $\{[a_n,b_n]\}_{n\neq m}\subset {\mathcal J}$.
\end{definition}

\begin{lemma}\label{fundamental lemma}
Suppose that $(f,{\mathcal I})$ is a fundamental pair with ${{\mathcal I}=\{[a_n,b_n]\}}$ and fundamental  envelope $P=P_{f,{\mathcal I},\{\varepsilon_n\}}$.  Let $\alpha>0$ and suppose that the point ${x_0\in (0,1)\backslash(\bigcup_{k=1}^\infty\{a_k,b_k\})}$,
 that is $x_0 \in (a_m,b_m)$ for some $m \in \field{N}$.  Then we can find a positive~$\delta$ with $[x_0-\delta,x_0+2\delta]\subset (a_m,b_m)$, a sequence $x_1=x_0+\delta,x_2,x_3,\ldots$  that is acceptable for $x_0$, and a fundamental pair $(g,{\mathcal J})$ with envelope $Q$ satisfying:
\begin{gather}
\label{g agrees f}
g=f \mbox{ on } [0,a_m] \cup [b_m,1],\\
\label{slope 1}
g \mbox{ is linear with slope } 1 \mbox { on }
[x_0-\delta,x_0] \mbox{ and on } [x_0+\delta,x_0+2\delta],\\
\label{acceptable g}
g \mbox{ is acceptable on }[x_0,x_1],\\
\label{alpha bound}
|g-f|< \alpha \mbox{ on } [0,1],\\
\label{J and I}
{\mathcal J}=\{I_n\}_{n \neq m} \cup \{[a_m,x_0-\delta],[x_0-\delta,x_0],[x_0+\delta,x_0+2\delta],\\ \nonumber [x_0+2\delta,b_m]\}\cup
 \{[x_{n+1},x_n]\colon n\in\field{N}\},\\
\label{Q subset P}
Q \subset P,\\
\text{ recalling that }S \text{ is the exceptional set from Theorem~\ref{countable theorem}
and also }\nonumber
\\
\text{appears in Definition~\ref{fundamental} we also have}\nonumber
\\
S\cap\left(\{x_0-\delta,x_0+\delta,x_0+2\delta\}\cup \{x_n: n\in \field{N}\}\right)=\emptyset.
\end{gather}
\end{lemma}

\begin{remark}\label{g is successor of f}
Note that it follows from (\ref{g agrees f}) and  (\ref{J and I}) that $(g,{\mathcal J})$ is a successor of $(f,{\mathcal I})$. Note also that $\delta$ can be chosen to be arbitrarily small.
\end{remark}

\begin{proof}[Proof of \cref{fundamental lemma}]
Assume that $f,{\mathcal I}, \{\varepsilon_n\}$, and $P$ are as in the statement of the \lcnamecref{fundamental lemma} and let $x_0 \in (a_m,b_m)$.  Then $(x_0,f(x_0))$ lies on the segment $L_m=L_{[a_m,b_m]}$ and is contained in the interior of $P_m=P_{L_m,\varepsilon_m}$.  It follows that we can choose $\delta>0$ small enough to ensure that
\begin{equation}\label{square containment}
[x_0-\delta,x_0+2\delta]\times [f(x_0)-\delta,f(x_0)+2\delta]\subset \interior{P_m}.
\end{equation}

We also require that $(\{x_0-\delta, x_0+\delta, x_0+2\delta\} \cup\bigcup_n\{x_n:n\in \field{N}\}) \cap S=\emptyset$.

Let $K_1$ and $K_2$ be the closed line segments connecting $(a_m,f(a_m))$ with $(x_0-\delta,f(x_0)-\delta)$ and $(x_0+2\delta,f(x_0)+2\delta)$ with $(b_m,f(b_m))$, respectively.   Note that $K_1$ and $K_2$ are contained in $\interior{P_m} \cup \{(a_m,f(a_m)),(b_m,f(b_m))\}$, and therefore we can choose $\varepsilon_0$ small enough so that $P_{K_i,\varepsilon_0} \subset P_m$ for $i=1,2$.
Now let $x_1=x_0+\delta$, choose $x_2,x_3,\ldots$ so that $\{x_n\}$ is acceptable for $x$ and define
\begin{equation}
g(x)=
\begin{cases}
            f(x) &  x\in[0,a_m]\cup[b_m,1], \\
            f(a_m)+\frac{f(x_0)-\delta-f(a_m)}{x_0-\delta-a_m}(x-a_m) & x \in [a_m,x_0-\delta],\\
            f(x_0)-x_0+x &  x \in [x_0-\delta,x_0],\\
            f(x_0)+f_{\{x_n\}}(x)   &  x \in [x_0,x_0+\delta],\\
            f(x_0)-x_0+x    & x \in [x_0+\delta,x_0+2\delta],\\
            f(x_0)+2\delta+\frac{f(b_m)-f(x_0)-2\delta}{b_m-x_0-2\delta}(x-x_0-2\delta) &  x \in [x_0+2\delta,b_m].
\end{cases}
\end{equation}

Then \cref{g agrees f,slope 1,acceptable g}
hold trivially and if $\delta >0$ is chosen small enough, we have \cref{alpha bound} as well.  Let ${\mathcal J}$ be defined by (\ref{J and I}).
It remains to choose our envelope $Q$ such that $Q\subset P$.  For each $k\in \field{N}$,
we let the set~$M_k$ be defined as $M_k=[(x_{k+1},g(x_{k+1})),(x_k,g(x_k))]$ and in case $k=0$, we define $M_0=[(x_0-\delta,f(x_0-\delta)),(x_0,f(x_0))]$ and for $k=\infty$, we set
\begin{equation*}
M_\infty =[(x_0+\delta,f(x_0+\delta)),(x_0+2\delta,f(x_0+2\delta))].
\end{equation*}
Note that for any $\varepsilon$ such that $0<\varepsilon<1$ and any $k \in \field{N}\cup\{0,\infty\}$ we have
\begin{equation*}
P_{M_k,\varepsilon}\subset [x_0,x_1]\times [g(x_0),g(x_1)]\subset \interior{P_m}\subset P.
\end{equation*}
Setting
\begin{equation*}
Q=\bigcup_{i=1,\,i\neq m}^\infty P_{L_i,\varepsilon_i} \cup (\bigcup_{i=1}^2 P_{K_i,\varepsilon_0}) \cup(\bigcup_{i=0}^\infty P_{M_i,1/2})\cup P_{M_\infty,1/2},
\end{equation*}
we see that $Q$ is an envelope for $g$ and $Q \subset P$, as desired.
\end{proof}

\begin{proof}[Proof of \cref{countable theorem}]
We begin by setting $f_0(x)=x$ on the interval~$[0,1]$ and ${\mathcal I}=\{[0,1]\}$ and letting ${P_0=P_{f_0,{\mathcal I},1/2}}$ be the envelope associated with $f_0$.  Note that $(f_0,{\mathcal I})$ is a fundamental pair.  Now using \cref{fundamental lemma} and \cref{g is successor of f} inductively
and recalling that $S=\{ s_{1},s_{2},\ldots \}$, it is easy to see that for each $n \in \field{N}$ we can choose $\delta_n>0$ (assume $\delta_n \convFromRightB 0$) and a fundamental pair $(f_n,{\mathcal I}_n)$ with envelope $P_n$ such that
\begin{gather}\label{f_n acceptable}
f_n \mbox{ is acceptable on } [s_n,s_n+\delta_n],\\
\label{f_n slope 1}
f_n \mbox{ is linear with slope $1$ on } [s_n-\delta_n,s_n] \mbox{ and on } [s_n+\delta_n,s_n+2\delta_n],\\
\label{s_n-delta_n not in S}
\{s_n-\delta_n,s_n+\delta_n,s_n+2\delta_n\}\cap S =\emptyset,\\
\label{2^-n bound}
|f_n-f_{n-1}|<2^{-n} \mbox{ on } [0,1],\\
\label{P_n containment}
P_n\subset P_{n-1},\\
\label{f_n successor}
(f_n,{\mathcal I}_n) \mbox{ is a successor of }(f_{n-1},{\mathcal I}_{n-1}).
	\end{gather}

We also require that for each $n \in \field{N}$ we have
\begin{gather}
[s_n-\delta_n,s_n+2\delta_n]\subset (a_m,b_m) \mbox{ where } [a_m,b_m]\in {\mathcal I}_{n-1}, \label{containment}\\
\mbox{ and }\quad [a_m,s_n-\delta_n],[s_n+2\delta_n,b_m]\in {\mathcal I}_n. \label{descendant}
\end{gather}

\begin{figure}
\centering
\includegraphics[scale=.8]{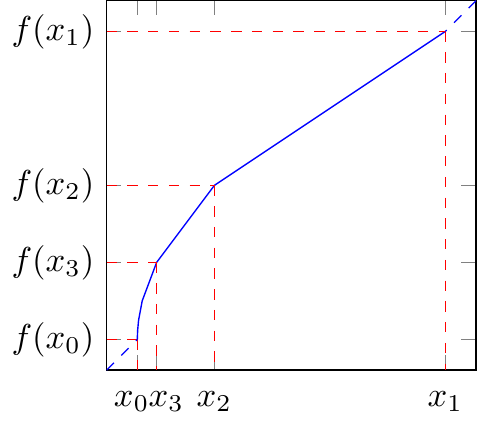}\quad\includegraphics[scale=.8]{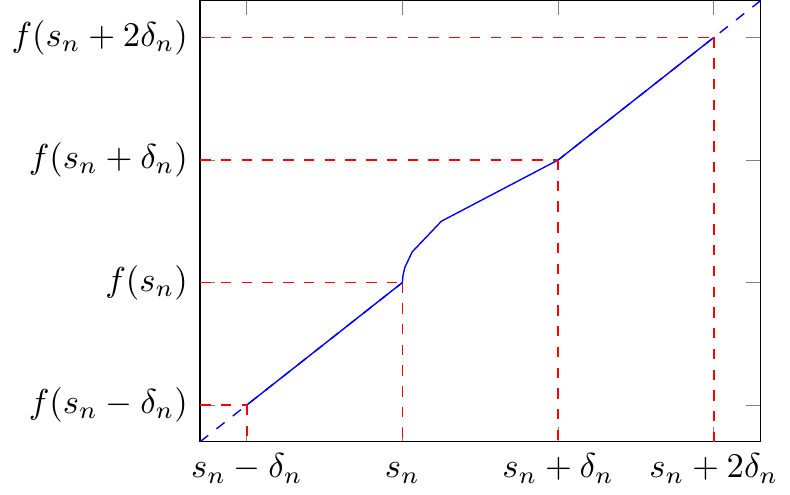}
\caption{Parts of $f_{x_0}$ and $f_n$, respectively}
\label{fundamental and fn}
\end{figure}

Using \cref{2^-n bound}, we can define $f$ as the pointwise limit of the sequence $\{f_n\}$ on $[0,1]$.    Clearly, $f$ is continuous and increasing on $[0,1]$.  We extend $f$ to the whole real line by defining $f(x)=x$ outside of $[0,1]$.   It remains to show that $l_f^\infty=S$ and that $\Lip f(x)<\infty$ for $x\not\in \overline{S}$.

Thus, suppose that $x \notin \overline{S}$.  If $x \notin [0,1]$, then clearly $\Lip f(x)=1< \infty$.  Similarly, if $x=0$ or $x=1$, then $\Lip f(x)<\infty$ since $f$ is linear on $\oped{-\infty}{0}$ and on $\clopen{1}{\infty}$ and the graph of $f$ restricted to $[0,1]$ is contained in $P_0$.  Thus we may assume that $x \in (0,1)$.  Now suppose that for some $n\in\field{N}$ the point $(x,f_n(x))$ is a vertex of one of the parallelograms that make up $P_n$.  In this case, since $x \notin S$ and $x \in (0,1)$, it follows that $(x,f_n(x))$ and hence $(x,f(x))$ is the shared vertex of two adjoining parallelograms in $P_n$.  Since the graph of $f$ restricted to $[0,1]$ is contained in $P_n$, it follows again that $\Lip f(x) < \infty$.  Thus, we may further assume that for each $n\in\field{N}$ the point $(x,f_n(x))$ is not a vertex of any parallelogram of $P_n$. Since $x \notin \overline{S}$, it follows from the construction of $f$, that $f$ is linear on an open interval containing $x$ and hence, once again, we have $\Lip f(x) < \infty$.

To finish the proof we need to show that $l_f^\infty=S$.  First of all, consider $s_n \in S$.  From \cref{P_n containment} it follows that the graph of $f|_{[0,1]}$ is contained in $P_n$ and therefore by \cref{f_n acceptable} and \cref{m_n goes to infinity}  it clearly follows that $\lip f(s_n)=\infty$.

   Now suppose that $x \notin S$. Since we have already established earlier that $\Lip f(x) < \infty$ if $x \notin (0,1)$, we may assume that $x \in (0,1)$. Similarly, we may also assume that for each  $n\in\field{N}$ the point $(x,f_n(x))$ is not a vertex of any parallelogram of $P_n$, $n \in \field{N}$.  For each $n \in \field{N}$ we define $J_n=(s_n-\delta_n,s_n+2\delta_n)$ and we consider two cases depending on whether or not $x$ is contained in finitely or infinitely many of these intervals.

Suppose, first of all, that $x$ is contained in only finitely many of these intervals $J_n$, and choose $N$ to be the largest $n$ such that $x \in J_n$.

  Recalling that for any $n$, $P_n$ is a countable union of non-overlapping parallelograms $\{Q_{n,j}\}$, we let $Q_n$ be the parallelogram in this collection containing the point $(x,f(x))$. For each $n \in \field{N}$ we also define the line segment~$L_n$ by $L_n=[(a_n, f_n(a_n)),(b_n, f_n(b_n))]$ and set $\varepsilon_n>0$ so that $Q_n=P_{L_n,\varepsilon_n}$. Then $x \in (a_n,b_n)$ for all $n \in \field{N}$.
	
	We note that for each $n\ge N$, the parallelogram~$Q_{n+1}$ is a direct descendant of $Q_n$. This fact follows from (\ref{P_n containment}), (\ref{containment}) and (\ref{descendant}).  Since the intervals $[a_n,b_n]$ are nested, we can let $a = \lim_{n\to \infty} a_n$ and $b=\lim_{n\to \infty} b_n$ so we have $a \le x \le b$.   If $a < x < b$, then the  graph of $f$ is linear on $[a,b]$ and it follows, as above, that $\Lip f(x)< \infty$, so trivially $\lip f(x) < \infty$.  Now suppose that either $x=a$ or $x=b$.  For each $n\in \field{N}$ define $r_n=\min\{x-a_n,b_n-x\}$.  Then, by
\cref{direct descendant inequality} and \cref{parallelogram estimate}, we have $q_f(x,r_n) \le q_{Q_N}$ for all $n \ge N$.  Since $r_n \to 0$, it follows that  $\lip f(x)< \infty$ in this case as well.

Finally, we assume that $x$ is contained in infinitely many of the intervals~$J_n$.  In particular, suppose that the point $x$ is contained in $J_n$.  If $x \in (s_n-\delta_n,s_n)\cup (s_n+\delta_n,s_n+2\delta_n)$, then, from \cref{f_n slope 1} and the fact that the graph of $f$ is contained in $P_n$, we get $q_f(x,r)\le 5/2$, as computed in \cref{q estimate}, where $r=\min\{|x-s_n+\delta_n|,|x-s_n|,|x-s_n-\delta_n|,|x-s_n-2\delta_n|\}$, so $r \le 1/2 \cdot \delta_n$.  On the other hand, suppose that
$x \in (s_n,s_n+\delta_n)$.  Using \cref{f_n acceptable}, \cref{f_n slope 1}, and that the graph of $f|_{[0,1]}$ is contained in $P_n$, we get
\begin{equation*}
f(s_n+2\delta_n)-f(s_n+\delta_n)=f(s_n+\delta_n)-f(s_n)=f(s_n)-f(s_n-\delta_n)=\delta_n,
\end{equation*}
and now it follows easily from this and the fact that $f$ is increasing that $q_f(x,\delta_n)\le 2$.

Thus, if $x \in J_n$, then $q_f(x,r)\le 5/2$ for some $r \le \delta_n$.  Therefore, if $x$ is contained in infinitely many intervals~$J_n$, we get $\lip f(x)<\infty$, and we are done with the proof of \cref{countable theorem}.
\end{proof}

A simple modification of the construction of $f$ in \cref{countable theorem} yields the following result.

\begin{corollary}\label{countable modification}
Given a countable set $S$ contained in an open interval $(a,b)$, there exists a continuous function $f\colon \field{R}\to \field{R}$ such that $l_f^\infty = S$, $\Lip f(x)<\infty$ for $x \notin \overline{S}$ and $f$ satisfies:
\begin{gather}\label{bounded by 1}
0 \le f(x) \le \min\{x-a,b-x\} \text{ for all } x \in (a,b),\\
f(x)=0 \mbox{ for all } x \notin (a,b)\label{zero outside (a,b)}.
\end{gather}
\end{corollary}

The goal in this \levelText{} is the proof of the following result.
\begin{proposition}\label{nowhere dense perfect case}
Let $E\subset \field{R}$ be a nowhere dense, perfect set. Then there exists a continuous function $f\colon \field{R}\to \field{R}$ such that $l_f^\infty=E$ and $f$ is constant on open intervals contiguous to $E$ and therefore $\Lip f(x)=0$ for all $x \notin E$.
\end{proposition}

In the proof of \cref{nowhere dense perfect case}, we will make use of the following simple \lcnamecref{lip and subsequences}, whose proof we leave to the reader.

\begin{lemma}\label{lip and subsequences}
Let $k>1$ and $f\colon\field{R}\to \field{R}$. Then
\begin{equation*}
  \lip f(x)=+\infty \Longleftrightarrow \lim_{n\to \infty} q_f(x,k^{-n})=\infty.
\end{equation*}
\end{lemma}

\begin{definition}[strongly intersecting]
An interval $J$ is said to \formatDefinition{strongly intersect} a set $F$ if $F\cap \interior{J}\not=\emptyset$.
\end{definition}

\begin{definition}[dyadic interval] Let $k \in \field{N}$.  We say the interval~$J$ is a \formatDefinition{dyadic interval of scale $4^{-k}$} if $J=[\frac{j}{4^k},\frac{j+1}{4^k}]$ for some $j \in \field{Z}$.
\end{definition}

The construction of $f$ in \cref{nowhere dense perfect case} uses the following \lcnamecref{dyadic 2 index}.  We leave the details of its proof to the reader.

\begin{lemma}\label{dyadic 2 index}
Let $E$ be a nowhere dense, bounded, perfect set. Denoting $m\coloneqq\min E$ and $M\coloneqq\max E$, we write $\displaystyle [m,M]\setminus E = \bigcup_{j=1}^\infty I_j$, where the intervals~$I_j$ are pairwise disjoint, open, and satisfy $\lvert I_j \rvert \ge \lvert I_{j+1} \rvert$. Then there is a sequence of integers ${0\eqqcolon h_0<h_1<h_2<\cdots}$ such that if $J$ is a dyadic
interval of scale $4^{-k}$ strongly intersecting $E$, then there are at least two indices $i$ such that $h_{k-1}<i\leq h_k$ and $I_i\subset J$.
\end{lemma}

\begin{proof}[Proof of \cref{nowhere dense perfect case}]
Assume that $E$ is nowhere dense and perfect.  We also assume without loss of generality that $E$ is bounded and define $m:=\min E$ and $M:=\max E$. We then choose a collection of open intervals $\displaystyle {\{I_j\}}_{j=1}^\infty={\{(a_j,b_j)\}}_{j=1}^\infty$ and a sequence of indices $\{h_j\}$ as in \cref{dyadic 2 index}.  Further, we define $I_{-1}=(-\infty,m)$ and $I_0=(M,\infty)$ and let $b_0=m$ and $a_0=M$.  After defining $f$ on $\field{R}\backslash (m,M)$ by setting
\begin{equation*}
f(x)=
\begin{cases}
  0 & \mbox{if } x\leq m, \\
  1 & \mbox{if } x\geq M,
\end{cases}
\end{equation*}
we next proceed by induction to define $f$ on each $\overline{I_j}=[a_j,b_j]$.  First, we define $f$ equal to 1/2 on $\overline{I_1}$.
Now suppose that we have defined $f$ on $\overline{I_1},\overline{I_2},\ldots, \overline{I_{i-1}}$ and assume that $h_{k-1}<i\leq h_k$. Let
\[ c_i=\max_{0\le j \le i-1,\, b_j < a_i}b_j \quad\text{ and  }\quad d_i=\min_{0\le j \le i-1,\, a_j>b_i}a_j.\]
In order to define $f$ on $\overline{I_i}=[a_i,b_i]$, we consider two cases.  First, assume that $I_i$ is contained in a dyadic interval $J$ of scale $4^{-k}$ and there is exactly one other interval $I_j$ with $1 \le j \le i-1$ such that $I_j \subset J$. If $\abs{f(c_i)-f(d_i)} < 2^{-k}$, then we define $f$ on $[a_i,b_i]$ as $f(c_i)+2^{-k+1}$. In all other cases, we assign the mean value of $f(c_i)$ and $f(d_i)$ to $f$ on $[a_i,b_i]$.

Our next task is to extend the definition of $f$ to the whole real line.  We begin with a pair of definitions and a couple of simple lemmas.

\begin{definition}[adjacent at level]\label{adjacent}
Suppose that $-1 \le i,j \le n$ and $b_i < a_j$.  We say that $I_i=(a_i,b_i)$ and $I_j=(a_j,b_j)$ are \formatDefinition{adjacent at level $n$} if none of the first $n$ intervals $I_1,I_2,\ldots, I_n$ are located between $I_i$ and $I_j$:
\begin{equation}\label{[b_i,a_j] intersect}
[b_i,a_j]\cap I_l =\emptyset \mbox{ for } l=1,2,\ldots,n.
\end{equation}
\end{definition}

\begin{definition}\label{J_k}
For each $k \in \field{N}$ we define ${\mathcal J}_k$ to be the collection of dyadic intervals $J$ of scale $4^{-k}$ such that $E \cap \interior{J}\neq \emptyset$.
\end{definition}

\begin{lemma}\label{8 interval}
If $I_i=(a_i,b_i)$ and $I_j=(a_j,b_j)$ are adjacent at level $n\ge h_k$, then $[b_i,a_j]$ intersects at most 8 distinct intervals from ${\mathcal{J}}_{k+1}$.
\end{lemma}

\begin{proof}
If the intervals $I_i$ and $I_j$ are adjacent at level $n$ where $n \ge h_k$, then it follows that $[b_i,a_j]$ does not contain any intervals from ${\mathcal J}_k$.

Therefore, $[b_i,a_j]$ intersects at most 2 distinct intervals from ${\mathcal J}_k$, and the \lcnamecref{8 interval} easily follows.
\end{proof}

\begin{lemma}\label{16 times 2^-k}
For each $j \in \field{N} \cup \{-1,0\}$ let $y_j$ be the value assigned to $f$ on $I_j$ so that $f(x)=y_j$ for all $x \in I_j$.
Suppose that $I_i=(a_i,b_i)$ and $I_j=(a_j,b_j)$ with $b_i<a_j$ are adjacent at level $n \ge h_k$.  Let
$\lambda = \min\{y_i,y_j\}$ and $\Lambda=\max\{y_i,y_j\}$.  Then for any $l \ge n$ such that $I_l\subset [b_i,a_j]$ we have
\begin{equation}\label{16 bound}
\lambda \le y_l \le \Lambda+16\cdot 2^{-k}.
\end{equation}
\end{lemma}

\begin{proof}
Assume that $I_i$ and $I_j$ as well as $l$ are as in the \lcnamecref{16 times 2^-k}.  Then $I_l$ has intervals $I_s$ and $I_t$ that are adjacent at level $l-1$ lying on its left and right, respectively.  Assume for the moment that $h_k < l \le h_{k+1}$.
Then either $y_l=\frac{y_s+y_t}2$ or
\begin{equation}\label{add 2^-k}
y_l=y_s+2^{-k}.
\end{equation}

But by \cref{8 interval}, equation \cref{add 2^-k} is applied at most 8 times for values of $l$ such that $n < l \le h_{k+1}$ and $I_l \subset [b_i,a_j]$.  It follows that $\lambda \le y_l \le \Lambda+8\cdot 2^{-k}$  if $n < l \le h_{k+1}$ and $I_l\subset[b_i,a_j]$.  Applying this same argument inductively on each interval $h_i < l \le h_{i+1}$ for $i\in\{k+1,k+2,\dots\}$, we get \cref{16 bound}.
\end{proof}

We now resume the proof of \cref{nowhere dense perfect case};  our next objective is to show that $f$ can be continuously extended to the whole real line. To that end, we pick an arbitrary point $x \in E$ and aim to show that the oscillation of $f$ at $x$ is $0$.  For each $i \in  \field{N}$ we define
\[ s_i=\max_{-1\le j \le i,\, b_j \le x}b_j \quad\text{ and }\quad t_i=\min_{0 \le j \le i,\, a_j \ge x}a_j.\]
We also define
\[ m_j=\inf_{t\in [s_j,t_j]\backslash E}f(t) \quad\text{ and }\quad M_j=\sup_{t\in [s_j,t_j]\backslash E}f(t).\]
Note that the nowhere denseness of~$E$ implies that
\begin{equation}\label{intersect = x}
\bigcap_{j=1}^\infty [s_j,t_j]=\{x\}.
\end{equation}
Therefore, in order to show that we may extend $f$ continuously at $x$, it suffices to prove that $M_j-m_j \to 0$ as $j \to \infty$.

To that end, we let $\varepsilon > 0$ and choose $k \in \field{N}$ so large that
\begin{equation}\label{17 less than epsilon}
17 \cdot 2^{-k} < \varepsilon.
\end{equation}
We next define an increasing sequence of integers
$1=j_1<j_2<j_3<\cdots$ inductively as follows:  For each $i \ge 1$ we let
$j_{i+1}$ be the smallest integer $j> j_{i}$ such that $ [s_j,t_j]\neq [s_{j_i},t_{j_i}]$.  Now, using \cref{intersect = x}, choose $i$ so that $t_{j_i}-s_{j_i}<4^{-k}$ and $j_i > h_k$. Suppose that
$\lvert f(t_{j_i})-f(s_{j_i})\rvert \ge 2^{-k}$.  In that case, $f$ is defined on $I_{j_{i+1}}$ as $\frac{f(s_{j_i})+f(t_{j_i})}2$, implying $\lvert f(t_{j_{i+1}})-f(s_{j_{i+1}}) \rvert =\frac12 \lvert f(t_{j_i})-f(s_{j_i})\rvert$.

Similarly, if  $\lvert f(t_{j_{i+1}})-f(s_{j_{i+1}}) \rvert \ge 2^{-k}$, then
\begin{equation*}
\lvert f(t_{j_{i+2}})-f(s_{j_{i+2}}) \rvert =\frac12 \lvert f(t_{j_{i+1}})-f(s_{j_{i+1}}) \rvert.
\end{equation*}  It follows that we can find $l > i$ such that
$\lvert f(t_{j_l})-f(s_{j_l}) \rvert < 2^{-k}$.  Applying \cref{16 times 2^-k} and inequality \cref{17 less than epsilon}, we get
$M_{j_n}-m_{j_n} < \varepsilon$ for all $n\ge l$, as desired.

We have now established that we can continuously extend $f$ to all of $\field{R}$.  Moreover, it follows from the construction that $f$ is constant on open intervals contiguous to $E$ so it remains to demonstrate that $\lip f(x)=\infty$ for all $x \in E$.

Assume that $x \in E$.  Using the fact that $E$ is perfect, we can choose a sequence of dyadic intervals $\{J_i\}$ such that each $J_i \in {\mathcal J}_i$ and $x\in J_i$.   Given an interval $J_i$, let $I_m=(a_m,b_m)$ and $I_n=(a_n,b_n)$ be two intervals in $J_i$; actually we want them to be the two intervals chosen first with this membership property.
We further assume that $b_m < a_n$.  From our rules for defining $f$ on the intervals $\{I_j\}$ it follows that
$\lvert f(b_m)-f(a_n)\rvert \ge 2^{-i-1}$.  Since $a_n, b_m, x \in J_i$, we see that $q_f(x,4^{-i})\ge \frac{2^{-i-2}}{4^{-i}}=2^{i-2}$.  Letting $i \to \infty$ and using \cref{lip and subsequences}, we end up with $\lip f(x)=\infty$, and we are done with the proof of \cref{nowhere dense perfect case}.
\end{proof}

\subsection{Countable union of nowhere dense, perfect sets}
Now, we start to look at unions of closed sets. To begin with, we look at countable unions of nowhere dense, perfect sets.
\begin{proposition}\label{nowhere dense perfect sets}
Suppose $F\subset \field{R}$ is a countable union of perfect, nowhere dense sets. Then there exists a continuous function $f\colon \field{R}\to \field{R}$ such that $f$ is constant on intervals contiguous to~$\overline{F}$ and $l_f^\infty=F$.
\end{proposition}

In order to prove \cref{nowhere dense perfect sets}, we will need the following \lcnamecref{E subset F}:

\begin{lemma}\label{E subset F}
Suppose that $E \subset F$, where $E$ is closed and $F$ is perfect, nowhere dense, and bounded.  Then there exists a collection $\mathcal{I}=\mathcal{I}_{E,F}$ of pairwise disjoint, closed intervals
$I=[a_I,b_I]$ satisfying:
\begin{gather}\label{F backslash E}
F\backslash E \subset \bigcup_{I \in \mathcal{I}}I,\\
I \cap E = \emptyset \ \ \text{for all}\ I \in \mathcal{I}\label{I intersect E},\\
F \cap I \text{ is perfect for all}\ I \in \mathcal{I}\label{perfect intersection},\\
\{a_I,b_I\}\subset F \ \ \text{for all}\ I \in \mathcal{I}\label{a_I and b_I}.
\end{gather}
Moreover, we can choose the intervals so that each closed subinterval of an interval contiguous to $E$ intersects only finitely many elements in $\mathcal{I}$.
\end{lemma}

\begin{proof}
Assume that $E$ and $F$ are as in the statement of the \lcnamecref{E subset F}.  If $F\backslash E$ is empty, then the result holds by taking the empty collection, so we assume that $F \backslash E \ne \emptyset$.  Let $\mathcal{J}$ be the collection of open intervals contiguous to $E$ that intersect $F$.  To prove the \lcnamecref{E subset F}, it suffices to  show that for each $J \in \mathcal{J}$ we can find a collection $\mathcal{I}_J$ of pairwise disjoint, closed intervals contained in $J$, which cover $F \cap J$ and such that for each $I \in \mathcal{J}$ the set $F\cap I$ is perfect  and the endpoints of $I$  lie in $F$. Additionally, we have to take care of the moreover-clause in the statement.

Assume that $J=(a,b)\in \mathcal{J}$. Let $c=\inf F\cap J$ and $d= \sup F \cap J$.  If $c \ne a$ and $d \ne b$, then we simply take $\mathcal{I}_J=\{[c,d]\}$.  Otherwise, suppose that $c \ne a$ and $d = b$.  In this case, using the fact that $F$ is perfect and nowhere dense, we can choose a sequence of open intervals $(c_1,d_1), (c_2,d_2), \ldots$ such that $c < c_n < d_n < c_{n+1}<d$ for $n=1,2,\ldots$ such that $d_n \to d$ and such that each $(c_n,d_n)$ is contiguous to $F$.  We then let
$\mathcal{I}_J=\{[c,c_1],[d_1,c_2],[d_2,c_3],\dots\}$.

It is easy to check that $\mathcal{I}_J$ has the desired properties in this case.  The cases where $c=a$ and $d \ne b$ and where $c=a$ and $d=b$ are handled similarly.
\end{proof}

Before setting out on the proof of \cref{nowhere dense perfect sets}, we state a few helpful definitions.

\begin{definition}[Functions $\Phi_I$ and $g_F$]\label{phi function} Given a bounded, open interval $I=(a,b)$, we define
\begin{equation*}
\Phi_I(x)=
\begin{cases}
   \min\{x-a,b-x\} & \mbox{if } x \in (a,b),\\
  0 & \mbox{if } x\notin (a,b).
\end{cases}
\end{equation*}
For intervals $I$ of the form $(-\infty,a)$ or $(a,\infty)$, where $a \in \field{R}$, we define
\begin{equation*}
\Phi_I(x)=
\begin{cases}
\abs{x-a} & \mbox{if } x \in I,\\
0  & \mbox{if } x \notin I.
\end{cases}
\end{equation*}
Given a bounded, nowhere dense perfect set $F\subset\field{R}$, we also define
\begin{equation*}
g_F=\sum_{I}\Phi_I,
\end{equation*}
where the sum is taken over all bounded intervals $I$ which are contiguous to $F$.
\end{definition}

\begin{proof}[Proof of \cref{nowhere dense perfect sets}]
First, given any bounded, nowhere dense perfect set $E$, we use \cref{nowhere dense perfect case} to find and fix a continuous function ${f_E\colon \field{R}\to \field{R}}$ such that $l_{f_E}^\infty=E$, the function $f_E$ is constant on intervals contiguous to $E$, vanishes on $\oped{-\infty}{\min(E)}\cup \clopen{\max(E)}{\infty}$, and $0 \le f_E \le 1$ on $\field{R}$.

One can readily prove the following useful observation:

\begin{observation}\label{g agrees with f}
Let $E$ be bounded, nowhere dense and perfect, and $f\colon\field{R}\to\field{R}$ be a continuous function that is constant on all intervals contiguous to~$E$ and satisfies $l^\infty_f = E$. Assume $g\colon\field{R} \to \field{R}$ satisfies the equality $g(x)=f(x)$ for all $x \in E$. Then we have
$\lip g(x)=\infty$ for all $x \in E$.
\end{observation}

Assume that $F=\bigcup_{n=1}^\infty F_n$, where each $F_n$ is nowhere dense and perfect.  We shall now construct a continuous function $f$ such that $l_f^\infty=F$.  We assume without loss of generality that each $F_n$ is bounded and that the sets $F_n$ are nested and differ: $F_n \subsetneq F_{n+1}$.

Now set $f_1=f_{F_1}$ and $g_1=g_{F_1}$.  We will construct $f$ in such a way that
${f_1 \le f \le f_1+g_1}$.  Since $g_1=0$ on $F_1$, \cref{g agrees with f} implies that $\lip f(x)=\infty$ for all $x  \in F_1$.

Using \cref{E subset F}, for each $n > 1$, we let
$\mathcal{I}_n=\mathcal{I}_{F_{n-1},F_n}$ be a pairwise disjoint collection of closed intervals $I=[a_I,b_I]$ satisfying
equations \cref{F backslash E,I intersect E,perfect intersection,a_I and b_I} with
$E=F_{n-1}$ and
$F=F_n$.  For each element $I \in \mathcal{I}_n$ we choose $O_I=(c_I,d_I)$ to be the open interval contiguous to $F_{n-1}$ that contains $I$ and define
$F_I=I \cap F_n$.  Then we choose $0<s_I\le \frac1{2^n}$ such that
\begin{equation}\label{s_I inequality}
s_I f_{F_I}+g_{F_I} \le \Phi_{O_I} \mbox{ on }I.
\end{equation}
Having already set $f_1=f_{F_1}$, we define for $n>1$
\begin{equation}\label{f_n and f_n-1}
f_n=\sum_{I \in \mathcal{I}_n}s_I f_{F_I}.
\end{equation}

Finally, we set $f=\sum_{n=1}^\infty f_n$.
Since each $f_n$ is continuous and $0 \le f_n \le \frac1{2^n}$ for $n > 1$, it follows that $f$ is continuous.  Note also that $f$ is constant on each open interval contiguous to  $\overline{F}$.  It remains to show that $l_f^\infty =F$.

We begin by showing that $\lip f(x)=\infty$ on $F$.   To that end, let
$x \in F$.  For notational convenience, we define
$\tilde{f}_n=\sum_{j=1}^n f_j$, so $f = \lim_{n \to \infty}\tilde{f}_n$. From \cref{s_I inequality} it follows that
\begin{equation*}
\tilde{f}_n\le\tilde{f}_{n+1}\le \tilde{f}_{n+1}+g_{{F}_{n+1}}\le
\tilde{f}_n+g_{{F}_n},
\end{equation*}
and therefore
\begin{equation}\label{tilde f_n}
\tilde{f}_n \le f \le \tilde{f}_n+g_{F_n}
\end{equation}
for all $n \in \field{N}$.  Notice also that for all $n$ and $k\in \field{N}$ we have that $f_{n+k}$ is 0 on ${F}_n$
and therefore $f=\tilde{f}_n$ on $F_n$.  Since $l_{\tilde{f}_n}^\infty=F_n$, it follows from \cref{g agrees with f} that
$\lip f(x)=\infty$ for all $x \in F$.

We are left with showing that $\lip f(x)<\infty$ for $x \in\overline{F}\setminus F$.  Assume that $x \in\overline{F}\setminus F$.  Then for each $n \in \field{N}$ there is an open interval $I_n=(a_n,b_n)$ that is contiguous to $F_n$ and that contains $x$.  Let $a=\lim_{n\to \infty} a_n$ and $b=\lim_{n\to \infty} b_n$ so we have
$a \le x \le b$.
Since $x\in \overline{F}$, we have either $a=x$ or $b=x$.  Note that $\tilde{f}_n$ is constant on $I_n$ and from \cref{tilde f_n}, we know that
$\tilde{f_n}\le f \le \tilde{f}_n+\Phi_{I_n}$ on $I_n$.  It follows that
$q_f(x,r_n)\le 2$ where $r_n = \min\{x-a_n,b_n-x\}$.

Since $r_n \to 0$, we get $\lip f(x)\le 2$, and we are done with the proof.
\end{proof}

\subsection{Meager \texorpdfstring{$F_{\sigma }$}{Fσ} sets}

In this section we improve \cref{nowhere dense perfect sets} by removing the requirement that the sets in the union be perfect. More precisely, we prove the following:

\newcommand{\pro}{Proposition}
\begin{proposition}\label{union of closed nowhere dense sets}
Let $F\subset \field{R}$ be a meager $F_{\sigma }$ set, that is it is the countable union of closed, nowhere dense sets. Then there is a continuous function $f\colon \field{R}\to \field{R}$ such that $l_f^\infty=F$ and such that $\Lip f$ is finite on $\field{R}\backslash \overline{F}$.
\end{proposition}

In order to accomplish our goal of proving \cref{union of closed nowhere dense sets}, we will need some preliminary results that allow us to write a given $F_\sigma$ set specifically tailored to our use. This is the content of the next section.

\subsubsection{Auxiliary results}\label{Auxiliary results}

\begin{lemma}\label{basic partition I}
Let $F$ be an $F_\sigma$ set that is not countable nor closed. Then there are countably many perfect, nowhere dense sets ${\{P_n\}}_{n\in \field{N}}$, and a countable set $C$ such that
\begin{gather}
F=\bigcup_{n\in \field{N}} P_n\cup \interior{F} \cup C,\notag\\
C\cap \interior{F}=C\cap P_n=P_n\cap \interior{F}=\emptyset \quad \text{for all $n\in \field{N}$}\label{emptyness}.
\end{gather}
\end{lemma}

\begin{proof}
As $F$ is an $F_\sigma$ set, there are countably many closed sets $F_n$ such that $F_n\subset F_{n+1}$ and
\begin{equation*}
F=\bigcup_{n\in \field{N}} F_n.
\end{equation*}

We let $O=\interior{F}$ be the interior of $F$. The sets $F_n\setminus O=F_n\cap (\field{R}\setminus O)$ are closed. By the Cantor--Bendixson theorem, there are   perfect sets $P_n$ and countable sets $C_n$ such that
\begin{equation*}
F_n\setminus O=P_n\cup C_n.
\end{equation*}
This means the following for $F$:
\begin{equation*}
F=\bigcup_{n\in \field{N}} (F_n\setminus O)\cup O=\bigcup_{n\in \field{N}} (P_n\cup C_n) \cup O.
\end{equation*}
We denote by $C$ the countable set $\bigcup_{n\in \field{N}}C_n\setminus (\bigcup_{n\in \field{N}} P_n)$. The equalities in \cref{emptyness} now follow from the definitions of the sets in question.
\end{proof}

In order to prove \cref{union of closed nowhere dense sets,F sigma}, we need a better version of  \cref{basic partition I}:

\begin{lemma}\label{basic partition}
Let $F$ be an $F_\sigma$ set that is not countable nor closed. Then there are countably many perfect, nowhere dense sets ${\{P_n\}}_{n\in \field{N}}$ and a countable set $D$ such that
\begin{gather}
F=\bigcup_{n\in \field{N}} P_n\cup \interior{F}\cup D,\notag\\
D\cap \interior{F}=D\cap \overline{\bigcup_{n\in\field{N}} P_n}=\bigcup_{n\in\field{N}}P_n\cap \interior{F}=\emptyset,\label{emptyness II}
\end{gather}
and the sets $P_n$ are nested for all $n\in \field{N}$, that is $P_n\subset P_{n+1}$.
\end{lemma}

\begin{proof}We assume that the representation of $F$ is already as in \cref{basic partition I}.  Let $C_1=C \cap(\overline{\bigcup_{n=1}^\infty P_n})$.
We only look at the case where $C_1$ is infinite and write $C_1=\{c_1,c_2,\dots\}$.  For each $c_n \in C_1$ we will construct a perfect, nowhere dense set $P_n'$ with the following properties:
\begin{gather}\label{Pn containment}
P_n \subset P_n'\subset F,\\
P_n'\subset C_1\cup\bigcup_{n=1}^\infty P_k\subset \overline{\bigcup_{k=1}^\infty P_k}\label{Pn' containment},\\
c_n \in P_n'\label{cn in pn'}.
\end{gather}

Defining $D=C\backslash C_1$ and replacing each $P_n$ with $\bigcup_{k=1}^n P'_k$, it is clear that \cref{basic partition} will follow from \cref{basic partition I} and
\eqref{Pn containment} to~\eqref{cn in pn'}.

We now proceed with the construction of $P_n'$.  Using that $c_n \in \overline{\bigcup_{k=1}^\infty P_k}$ and that each $P_k$ is perfect and nowhere dense, we find a subsequence $\{P_{k_i}\}$ of $\{P_k\}$ and intervals $I_j=[a_j,b_j]$, such that
\begin{gather}\label{I_j pairwise disjoint}
I_j \cap I_i =\emptyset \text{ if } j \ne i,\\
I_j \cap P_{k_j} \text{ is perfect}\label{I_j intersect perfect},\\
a_j \to c_n\label{aj limit cn}.
\end{gather}

Defining $P_n' = P_n \cup(\bigcup_{j=1}^\infty (I_j \cap P_{k_j}))\cup \{c_n\}$, it is straightforward to check that $P_n'$ is perfect, nowhere dense, and that  \eqref{Pn containment} to \cref{cn in pn'} all hold, completing the proof.
\end{proof}

\subsubsection{Proof of \texorpdfstring{\cref{union of closed nowhere dense sets}}{\pro~\ref{union of closed nowhere dense sets}}}
\begin{proof}[Proof of \cref{union of closed nowhere dense sets}]
Assume that $F$ is the countable union of closed, nowhere dense sets.  Then, by \cref{basic partition}, we have
\begin{equation*}
  F=(\bigcup_{n=1}^\infty P_n) \cup D,
\end{equation*}
where each $P_n$ is perfect and nowhere dense, $D$ is countable and the intersection
${\overline{\bigcup_{n=1}^\infty P_n}\cap D}$ is empty.  Let $P=\overline{\bigcup_{n=1}^\infty P_n}$.  Then we can write $\field{R}\backslash P$ as the countable union of pairwise disjoint open intervals: $\field{R}\backslash P =\bigcup_{n=1}^\infty (a_n,b_n)$.  Using \cref{countable modification} for each $n \in \field{N}$ we find a continuous $h_n \colon \field{R}\to \field{R}$ such that
\begin{gather}\label{l_f_n and D}
l_{h_n}^\infty = D \cap (a_n,b_n),\\
0 \le h_n(x) \le \min\{x-a_n,b_n-x\} \text{ for all }x \in (a_n,b_n),\label{f_n inequality}\\
\label{Lip h_n finite} \Lip h_n \text{ is finite on } (a_n,b_n)\backslash\overline{D},
\intertext{and}
h_n(x)=0 \text{ for all }x \notin (a_n,b_n).\label{h_n zero}
\end{gather}
We let $h=\sum_{n=1}^\infty h_n$ and note that it follows from \cref{l_f_n and D,f_n inequality,Lip h_n finite,h_n zero} that $h$ is continuous on $\field{R}$, $l_h^\infty =D$, and $\Lip h$ is finite on $\field{R}\backslash\overline{D}$.   Now, using \cref{nowhere dense perfect sets}, we choose a continuous function $g \colon \field{R}\to \field{R}$ satisfying $l_g^\infty = \bigcup_{n=1}^\infty P_n$ and that $\Lip g$ is finite on $\field{R}\backslash \overline{\bigcup_{n=1}^\infty P_n}$ and therefore finite on $\field{R}\backslash \overline{F}$.

We claim that $f=g+h$ has the desired properties.

First note that since $\field{R}\backslash\overline{F}\subset \field{R}\backslash\overline{D}$, it follows that $\Lip h$ is finite on $\field{R}\backslash\overline{F}$ and thus, both $\Lip h$ and $\Lip g$ are finite on $\field{R}\backslash\overline{F}$.   Therefore, $\Lip f$ is finite on $\field{R}\backslash\overline{F}$, as required.

    To conclude the proof we need to show that $l_f^\infty =F$. Since $g$ is constant on each $(a_n,b_n)$, and each $h_n$ is constant on $\field{R}\backslash(a_n,b_n)$, it follows from \cref{l_f_n and D} that $l_f^\infty \cap (\field{R}\backslash P)=D$.
It remains to verify that $l_f^\infty \cap P = \bigcup_{n=1}^\infty P_n$. This follows easily from the fact that $l_g^\infty=\bigcup_{n=1}^\infty P_n$ and that $\Lip h(x)<\infty $ on $\field{R}\backslash (\bigcup_{n=1}^\infty (a_n,b_n))$; consequently we have finished the proof.
\end{proof}

\subsection{Union of closed sets}
The proof of \cref{F sigma} now follows rather easily from \cref{union of closed nowhere dense sets} and the following two lemmas.

\begin{lemma}\label{lip infinite on open interval}
Given any open interval $(a,b)$ and $h>0$ there exists a continuous function $f\colon \field{R}\to\field{R}$ such that

\begin{gather}\label{f is zero}
f=0 \text{ on } \field{R}\backslash (a,b),\\
0 \le f(x) \le h\cdot \min\{x-a,b-x\} \text{ for } x \in (a,b)\label{f is dominated},\\
\lip f(x)=\infty \text{ for all } x \in (a,b).\label{lip f is infinite}
\end{gather}
\end{lemma}

Note that if $f$ is as in the lemma, then we also have that $\Lip f$ is finite on $\field{R}\backslash (a,b)$.

\begin{proof}[Proof of~\cref{lip infinite on open interval}]
We start with a definition: If $f$ is linear on $[a,b]$ with $f(a)=c$ and $f(b)=d$, then we define $f^{[a,b]}\colon [a,b] \to \field{R}$ so that
$f^{[a,b]}$ is linear on each of the intervals $[a,a+\frac{b-a}3],[a+\frac{b-a}3,a+\frac{2(b-a)}3]$ and $[a+\frac{2(b-a)}3,b]$ and so that
\begin{equation*}
f^{[a,b]}(a)=f^{[a,b]}\biggl(a+\frac{2(b-a)}3\biggr)=c \text{ and } f^{[a,b]}\biggl(a+\frac{b-a}3\biggr)=f^{[a,b]}(b)=d.
\end{equation*}
We next define an auxiliary function $g$ on $[0,1]$ as follows:

We begin by setting $g_0(x)=x$ on $[0,1]$.  For each $n \in \field{N}$ we define
\begin{equation*}
I_{n,j}=\biggl[\frac{j}{3^n},\frac{j+1}{3^n}\biggr] \text{ for } j\in\{0,1,\dots,3^n-1\}.
\end{equation*}
Next we define a sequence of functions $\{g_n\}$ recursively on $[0,1]$ so that
\begin{equation*}
g_n|_{I_{2n-1,j}}=g^{I_{2n-1,j}}_{n-1} \text{ for all } n \in \field{N} \text{ and } j\in\{0,1,\dots, 3^{2n-1}-1\}.
\end{equation*}
Then since each $g_n$ is continuous on $[0,1]$ and $\norm{g_n-g_{n-1}}\le 2\cdot 3^{-n}$ for each $n \in \field{N}$, it follows that
$g=\lim_{n\to \infty} g_n$ is continuous on $[0,1]$.  It is also easy to verify that $q_g(x,3^{-2n})\ge 3^n/2$ for each $x \in (0,1)$ and for each $n \in \field{N}$.  It follows from \cref{lip and subsequences} that $\lip g=\infty$ on $(0,1)$.  Finally, given an interval $[a,b]$, we define
\begin{equation*}
f(x)=
\begin{cases}
   h\Phi_{(a,b)}(x) \frac{1+g(\frac{x-a}{b-a})}2 & \mbox{if } x \in (a,b),\\
  0 & \mbox{if } x\notin (a,b).
\end{cases}
\end{equation*}
It is easy to verify that $f$ has the desired properties.
\end{proof}

\begin{lemma}\label{open set case}
Given an open set $O \subset \field{R}$, there exists a continuous function $f \colon \field{R} \to \field{R}$ such that
$l_f^\infty =O$ and $\Lip f(x) < \infty$ for all $x \in \field{R} \backslash O$.
\end{lemma}

\begin{proof}
Let $O$ be equal to $\bigcup_n (a_n,b_n)$, where the intervals $(a_n,b_n)$ are pairwise disjoint.  For each $n$ we construct a continuous function $f_n \colon \field{R} \to \field{R}$ satisfying \cref{f is zero,f is dominated,lip f is infinite} with $a$ and $b$ replaced with $a_n$ and $b_n$, $f$ replaced with $f_n$ and $h$ replaced by $1/2$. Letting $f=\sum_n f_n$ it is easy to verify that $l_f^\infty=O$ and $\Lip f(x) \le 1$ for all $x \in \field{R}\backslash O$.
\end{proof}

\begin{proof}[Proof of \cref{F sigma}]
Let $F$ be a countable union of closed sets and define $O=\interior{F}$ and $E=F\backslash O$.  It follows that $E$ is a countable union of nowhere dense closed sets, so by \cref{union of closed nowhere dense sets} there exists a continuous function $g$ such that $l_g^\infty=E$ and $\Lip g$ is finite on $\field{R}\backslash \overline{E}$.  Moreover, using \cref{open set case}, we can find a continuous function $h$ such that $l_h^\infty=O$ and $\Lip h$ is finite on $\field{R}\backslash O$.  Let $f=g+h$.  Since $\lip g=\infty$ on $E$ and $\Lip h<\infty$ on
$E \subset \field{R}\backslash O$, we see that $\lip f=\infty$ on $E$.  Similarly, since $\Lip g <\infty$ on $O \subset \field{R}\backslash \overline{E}$ and $\lip h=\infty$ on $O$, we have $\lip f=\infty$  on $O$.  Finally, since $\lip g < \infty$ on $\field{R}\backslash F \subset \field{R}\backslash E$ and $\Lip h<\infty$ on $\field{R}\backslash F \subset \field{R}\backslash O$, we have $\lip f < \infty$ on $\field{R}\backslash F$, finishing the proof.
\end{proof}

\subsection{The \texorpdfstring{$G_\delta$}{Gδ} case}
In this section, we prove the following result:

\begin{theorem}\label{g delta theorem}
For every $G_\delta$ set $E\subset\field{R}$, there exists a continuous function $f$ such that $L_f^\infty=l_f^\infty=E$.
\end{theorem}

We begin with a few definitions:

First of all, for each interval $I$ and each $n \in \field{N}$ we define the function~$\Phi_{I,n}$ by $\Phi_{I,n}(x)=\Phi_I(x)/5^n$, where $\Phi_I$ is as previously defined in \cref{phi function}.
Let $I=(a,b)$ be an interval and $n \in \field{N}$.  A basic building block in the construction of the function $f$ advertised in \cref{g delta theorem}  will be the function $\Psi_{I,n}(x)$ that we now proceed to define.

We begin by defining a countable set of points
\begin{equation}\label{*bin*}
B_{I,n}={\{x_k\}}_{k\in \field{Z}}
\end{equation}
 contained in $I$.  First set $x_0=\frac{a+b}2$ and $x_1=\frac{a+b}2+\frac{b-a}{2\cdot 5^{2n}}$.  For $k \in \field{N}$, let $x_{2k+1}=x_{2k-1}+\frac{b-a}{5^{2n}}(\frac{5^{2n}-1}{5^{2n}+1})^k$ and $x_{2k}=\frac{x_{2k-1}+x_{2k+1}}2$ and set $x_{-k}=a+b-x_k$.  A bit of calculation should now convince the reader that $x_k < x_{k+1}$ for all $k\in\field{Z}$,  $\lim_{k\to \infty}x_k=b$ and
$\lim_{k\to \infty}x_{-k}=a$.

Now define $\Psi=\Psi_{I,n}\colon\field{R}\to \field{R}$ to be the unique continuous function with the following properties:
$\Psi(x)=0$ for all $x \notin I$, $\Psi(x_{2k+1})=0$ for all $k \in \field{Z}$, the function~$\Psi$ is linear with slope $5^n$ on each interval
$[x_{2k+1},x_{2k+2}]$ and linear with slope $-5^{n}$ on each interval $[x_{2k},x_{2k+1}]$.  Note that the graph of $\Psi|_I$ consists of countably many straight line segments of slope $\pm 5^n$.

\begin{figure}
\centering
\includegraphics{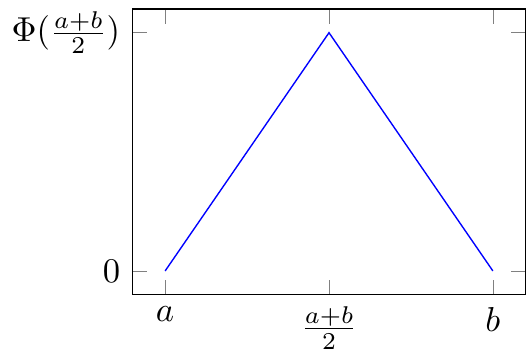}
\quad
\includegraphics{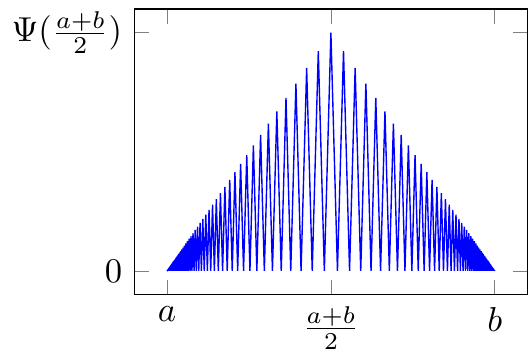}
\caption{$\Phi_{I,n}$ and $\Psi_{I,n}$}
\label{Phi and Psi}
\end{figure}

One can also verify that
\begin{equation}\label{psi phi inequality}
0 \le \Psi_{I,n}(x)\le \Phi_{I,n}(x) \mbox{ for all } x \in \field{R}.
\end{equation}

Additionally, we have
\begin{equation}\label{psi phi equality}
\Psi_{I,n}(x_{2k})=\Phi_{I,n}(x_{2k}) \text{ for all } k \in \field{Z}.
\end{equation}

For future reference we record the following useful observations:
\begin{gather}\label{max x_j+1-x_j}
\max_j(x_{j+1}-x_j)=x_1-x_0=\frac{b-a}{2\cdot5^{2n}},\\
\frac{12}{13}\le\frac{5^{2n}-1}{5^{2n}+1}\le\frac{x_{j+2}-x_{j+1}}{x_{j+1}-x_j}\le\frac{5^{2n}+1}{5^{2n}-1}\le\frac{13}{12}\label{x_j x_j+1 ratio},\\
\abs{\frac{\Psi(x_{2j+2})}{\Psi(x_{2j})}}=
\begin{cases}
            \frac{5^{2n}-1}{5^{2n}+1} &  j\ge 0, \\
            \frac{5^{2n}+1}{5^{2n}-1} & j<0.
\end{cases}\label{Psi ratio}
\end{gather}

Note that from  (\ref{psi phi inequality}), (\ref{psi phi equality}), and the definition of $\Phi_{I,n}$ we get
\begin{equation}\label{sup Psi_n}
\Psi_{I,n}(x) \le \Psi_{I,n}(x_{2j})+\frac{\abs{x-x_{2j}}}{5^n} \text{ for all }x \in I.
\end{equation}

Given a collection of open sets $\{G_1,G_2,\ldots\}$ for each $n \in \field{N}$ we let ${\mathcal I}_n$ be the unique collection of pairwise disjoint open intervals whose union is $G_n$ so $G_n=\bigcup_{I \in {\mathcal I}_n}I$.  Given such a collection,
by using \eqref{*bin*} and the definition after it we define:
\begin{equation}\label{X_n}
X_n=\bigcup_{I \in {\mathcal I}_n}B_{I,n}.
\end{equation}

The following simple \lcnamecref{sparse G delta} is easily proved by induction.

\begin{lemma}\label{sparse G delta}
Let $E$ be a $G_\delta$ set with empty interior.  Then we can find open sets $G_1,G_2,\ldots$
such that
\begin{gather}\label{E=intersection}
E=\bigcap_{n=1}^\infty G_n,\\
G_{n+1}\subset G_n \mbox{ for all }n \in \field{N}\label{G_n contains G_n+1},\\
\sup_{I\in {\mathcal I}_1} |I| \le 1\label{sup I},\\
\mbox{and each }I\in {\mathcal I}_{n+1} \mbox{ contains at most one point of }X_n.\label{1 point}
\end{gather}
\end{lemma}

\begin{proof}[Proof of \cref{g delta theorem}] We first consider the case where $E$ is a $G_\delta$ set with empty interior.  Assume that the sets $G_1,G_2,\ldots$ have been chosen as in \cref{sparse G delta}.  For each $n \in \field{N}$ define
\begin{equation}\label{f_n def}
f_n(x)=\sum_{I\in {\mathcal I}_n} \Psi_{I,n}(x)
\end{equation} and let
\begin{equation*}
f(x)=\sum_{n=1}^\infty f_n(x).
\end{equation*}
It follows from \cref{max x_j+1-x_j,G_n contains G_n+1,1 point,sup I,f_n def}, and the definition of
$\Psi_{I,n}$ that $f$ is continuous, so we are left to show that $\lip f(x)=\infty$ for all $x \in E$ and $\Lip f(x)<\infty$ for all $x \notin E$.

First we show that $\Lip f(x) < \infty$ for all $x \notin E$.  Assume that $x \notin E$.  Then we can choose $K \in \field{N}$ such that $x \notin G_n$ for all $n>K$,  and we write $f=S_K+T_K$,  where $S_K=\sum_{n=1}^K f_n$ and $T_K=\sum_{n=K+1}^\infty f_n$.  Since each $f_n$ is $5^n$\nobreakdash-Lipschitz, it follows that $S_K$ is Lipschitz and therefore $\Lip S_K(x)<\infty$.  Thus, to establish that $\Lip f(x)<\infty$ it suffices to prove the inequality
$\Lip T_K(x)<\infty$.  Now for each $n \in \field{N}$ we define
\begin{equation*}
h_n=\sum_{I \in {\mathcal I}_n}\Phi_{I,n},
\end{equation*}
so we have, by the fact that $\Phi_{I,n}$ bounds $\Psi_{I,n}$ from above as detailed in \cref{psi phi inequality}, that
$0 \le f_n(t) \le h_n(t)$ for all $t \in \field{R}$ and therefore
\begin{equation}\label{T_K inequality}
0 \le T_K(t)\le \sum_{n=K+1}^{\infty}h_n(t)=:R_K(t) \mbox{ for all }t \in \field{R}.
\end{equation}
That $h_n$ is $5^{-n}$\nobreakdash-Lipschitz implies that $R_K$ is Lipschitz.   It therefore follows from \cref{T_K inequality} and the fact that $0=T_K(x)=R_K(x)$ that $\Lip T_K(x)<\infty$ as desired.

It remains to show that $\lip f(x)=\infty$ for all $x \in E$.  Let $x \in E$ and assume without loss of generality that $x=0$.  For each $n \in \field{N}$ we choose $I_n \in {\mathcal I}_n$ such that $0 \in I_n$ and we define $\Psi_n=\Psi_{I_n,n}$.  We next write $f=g+h$ where
\begin{equation*}
g=\sum_{n=1}^\infty \Psi_{n} \mbox{ and } h=\sum_{n=1}^\infty \sum_{I\in{\mathcal I}_n,\,I\neq I_n}\Psi_{I,n}.
\end{equation*}
Using the same argument as above, one can show that $\Lip h(0)<\infty$, and therefore it suffices to prove that $\lip g(0)=\infty$.

For each $n \in \field{N}$ choose $x^n_j,\,x^n_{j+1}\in B_{I_n,n}=:B_n$ so that $0 \in [x^n_j,x^n_{j+1}]$, choose $J_n$ to be the larger of the two intervals $[x^n_j,0]$ and $[0,x^n_{j+1}]$ and define $r_n=|J_n|$.  Note that
\begin{equation}\label{J_n 1/2}
r_n\ge\frac12(x^n_{j+1}-x^n_j)
\end{equation}
and $(\interior{J_n})\cap B_n=\emptyset$. Using these facts along with \cref{max x_j+1-x_j,x_j x_j+1 ratio,1 point} we get
\begin{equation}\label{J_n J_n+1}
2\cdot 5^{2n}r_{n+1}\leq \abs{I_{n+1}}\leq{5r_n}.
\end{equation}

We also note that $\Psi_n$ is linear on $J_n$ with slope $\pm 5^n$. Define $m_n$ by the equality $m_n=\max\{\Psi_n(x_j^n),\Psi_n(x_{j+1}^n)\}$ and note that from (\ref{J_n 1/2}) we get $m_n \le2\cdot 5^n r_n$.  It follows from (\ref{sup Psi_n}) that
\begin{equation}\label{Psi_n domination}
\Psi_n(t)\le 2\cdot 5^n r_n+\frac1{5^n}(|t|+r_n) \text{ for all }t.
\end{equation}
Let $g_k=\sum_{j=1}^k \Psi_j$.  Since each $\Psi_j$ is $5^j$\nobreakdash-Lipschitz, it follows that $g_k$ is $\frac{5^{k+1}}4$\nobreakdash-Lipschitz.
Therefore,  since $g_n=\Psi_n+g_{n-1}$ and $\abs{\Psi_n'}=5^n$ on $J_n$,  we get
\begin{equation}\label{g_n estimate}
\abs{g_n(t)-g_n(s)}\ge \frac34\cdot 5^n\abs{t-s} \mbox{ for all }s,t \in J_n.
\end{equation}

Now we define $[a_n,b_n]=J_n$ and $q(r)=q_g(0,r)$ and further note that ${q(r_n)\ge \frac{|g(b_n)-g(a_n)|}{r_n}}$.

We need to show that $\lim_{r\convFromRight 0} q(r)=\infty$. We begin by showing that ${q(r_n)\to \infty}$ as $n \to \infty$.  From \cref{J_n J_n+1} we get

\begin{equation}\label{Psi_n inequality}
\sup_{t \in \field{R}}\Psi_n(t)\le \frac1{2\cdot 5^n}|I_n|\le\frac1{2\cdot 5^{n-1}}r_{n-1},
\end{equation}
and it follows that
\begin{equation}\label{sum Psi_k}
\sum_{k>n}\Psi_k(t)\le \frac1{5^n}r_n \ \text{for all}\ t \in \field{R}.
\end{equation}

But from \cref{g_n estimate} we know that $|g_n(b_n)-g_n(a_n)|\ge \frac{3}{4}\cdot 5^n r_n$, so it follows that $|g(b_n)-g(a_n)|\ge \frac{1}{2}\cdot 5^n r_n$, and therefore $q(r_n)\ge \frac{1}{2}\cdot 5^n$. This implies that $q(r_n)$ converges to $\infty$.

We now consider the interval $[r_{n+1},r_n]$.  We need to show that $q(r)$ remains large on the entire interval and not just at the endpoints, $r_{n+1}$ and $r_n$.  Note that for any $r>0$ and $s\ge 1$ we have $q(sr)\ge q(r)/s$.

We assume without loss of generality that $a_n=0$ so $J_n=[0,r_n]$.  Then from \cref{g_n estimate} we have
\begin{equation}\label{g_n zero estimate}
|g_n(t)-g_n(0)| \ge \frac34\cdot 5^n t \mbox{ for all } t \in [0,r_n].
\end{equation}
Now, using (\ref{Psi_n domination}), we get
\begin{equation}\label{Psi_n+1 inequality}
\Psi_{n+1}(t)\le \frac{5^n}2t  \mbox{ for all }t \ge 25r_{n+1}.
\end{equation}

Finally,  using \cref{sum Psi_k}, \cref{g_n zero estimate}, and \cref{Psi_n+1 inequality}, we get $|g(t)-g(0)|\ge \frac{1}{5}\cdot5^n t $ on $[25r_{n+1},r_n]$ and therefore $q(r)\ge 5^{n-1}$ on $[25r_{n+1},r_n]$.  It follows that $\lim_{r \convFromRight 0}q(r)=\infty$ so $\lip g(0)=\infty$, as desired.

To complete the proof of \cref{g delta theorem}, we need to remove the assumption that $E$ has empty interior.  Thus we now assume that $E$ is a $G_\delta$ set with non-empty interior.  Let $O=\interior E$ and $E'=E\backslash O$.  Then $E'$ is a $G_\delta$ set and using what we have proved so far we can find a function $h$ such that $\lip h(x)=\infty$ on $E'$ and $\Lip h(x) < \infty$ on $\field{R}\backslash E'$.

In \cref{open set case}, we constructed a continuous function~$g$ satisfying the equality $\lip g(x)=\infty$ on $O$ and the inequality $\Lip g(x)<\infty$ on the complement of $O$.
Setting $f=g+h$ we have the result we need.
\end{proof}

\section{Typical Results}\label{Typical Results}
In this section we consider continuous functions defined on $\field{R}^d$, where $d$ is a positive integer.  Our goal here is to show that the typical continuous function has vanishing $\lip$ at points of a residual set of full measure. We start with some auxiliary results.
\begin{notation}[balls $B_\infty(x,r)$ and $U_\infty(x,r)$ and Lebesgue measure $\abs{E}$]
Let $d\in\field{N}$, choose $x=(x_1,\dots, x_d)\in \field{R}^d$ and $r>0$. We denote by $B_\infty(x,r)$ and $U_\infty(x,r)$ the closed and open balls with respect to the maximum norm on $\field{R}^d$, centered at $x$ and with radius $r$, respectively. We denote the $d$\nobreakdash-dimensional Lebesgue measure of a set $E\subset\field{R}^d$ by $\abs{E}$.
\end{notation}

\begin{theorem}\label{typical function}
The typical function $f\in C({[0,1]}^d)$ satisfies $\lip f(x)=0$ at points of a residual set of full measure in ${[0,1]}^d$.
\end{theorem}

\begin{proof}
We aim to find a residual set of functions in $C({[0,1]}^d)$ such that each member~$f$ satisfies $\lip f(x)=0$ at all points of a set of full measure (depending on~$f$) in ${[0,1]}^d$; the residuality part then follows automatically, as for any continuous function $f$, the set of points $x$ with $\lip f(x)=0$ is of the type $G_\delta$ by a simple multidimensional
generalization of
\cref{Baire classes}
(and, of course, full measure implies density).

First, for each $n \in \field{N}$ set
\begin{equation}\label{alpha and beta}
\alpha_n:=\sqrt[d]{1-2^{-n}}, \quad \beta_n:= \frac{1-\alpha_n}{n},
\end{equation}
and define ${\mathcal C}_n$ as the set of all finite sequences ${(C_i)}_{i=1}^k$ (where $k$ can be any natural number) of pairwise disjoint closed cubes in ${[0,1]}^d$ (i.e.\ closed balls in the maximum norm on $\field{R}^d$) such that for each $i\in\{1,\dots, k\}$ the side length $l(C_i)$ of $C_i$ is less than $1/n$, and $\abs{ \bigcup_{i=1}^k C_i} > 1-2^{-n}$.
Finally, for each $n\in\field{N}$ we define
\begin{multline*}
A_n := \big\{ f\in C({[0,1]}^d) : \exists\{C_i\}_{i=1}^k \in \mathcal{C}_n \quad \exists \{a_i\}_{i=1}^k\subset \field{R} \\ \forall i\in \{1,\dots,k\}, \quad \forall x\in C_i :  \abs{f(x)-a_i} < \beta_n l(C_i)\big\}.
\end{multline*}

The proof will be finished when we prove the following two claims:
\begin{alist}
\item\label{Intersection is residual} The intersection $A:= \bigcap_{n=1}^\infty A_n$ is residual in $C({[0,1]}^d)$; in fact, $A_n$ is open and dense for each $n$.
\item\label{lip zero almost everywhere} If $f\in A$, then $\lip f(x)=0$ for almost every $x\in {[0,1]}^d$.
\end{alist}

To prove claim~\cref{Intersection is residual}, we first observe that $A_n$ is open for any $n$; to that end, fix $n$ and $f\in A_n$. Take $\{C_i\}_{i=1}^k\in {\mathcal C}_n$ and $\{a_i\}_{i=1}^k\subset \field{R}$ witnessing the fact that $f\in A_n$. Thus we have for each $i\in\{1,\dots,k\}$ and each $x\in C_i$ that
\begin{equation*}
\abs{f(x)-a_i}< \beta_n l(C_i).
\end{equation*}
But the finitely many cubes $C_i$ are compact, so there exists $\gamma>0$ (depending only on $\{C_i\}_{i=1}^k$) such that
\begin{equation*}
\abs{f(x)-a_i}< \beta_n l(C_i) -\gamma
\end{equation*}
for each $x\in C_i$. Now, if $g\in C({[0,1]}^d)$ is such that $\|f-g\|_\infty < \gamma$, then for each $i$ and each $x\in C_i$ we have
\begin{equation*}
\abs{g(x)-a_i}\leq \abs{g(x)-f(x)} + \abs{f(x)-a_i} < \gamma + \beta_n l(C_i) - \gamma = \beta_n l(C_i).
\end{equation*}
Thus, $g\in A_n$, and $A_n$ is open.

To prove the density of $A_n$ in $C({[0,1]}^d)$, let there be given an arbitrary function \mbox{$f\in C({[0,1]}^d)$} and $\varepsilon>0$; we want to find a function $g\in A_n$ such that $\|f-g\|_\infty < \varepsilon$. From the uniform continuity of $f$ we obtain a $\delta>0$ such that for each $x\in {[0,1]}^d$ and each $y \in B_\infty(x,\delta)\cap{[0,1]}^d$ we have $\abs{f(x)-f(y)} < \varepsilon/2$. Next, let us find $\{x_i\}_{i=1}^k \subset {[0,1]}^d$ and $\{r_i\}_{i=1}^k \subset (0,\delta)$ such that $\{B_\infty(x_i,r_i)\}_{i=1}^k \in {\mathcal C}_n$ (clearly we can do that; recall that $B_\infty$ denotes the closed ball in the maximum norm), and take a number $\gamma\in (0,1)$ so close to $1$ that also $\{B_\infty(x_i,\gamma r_i)\}_{i=1}^k \in {\mathcal C}_n$. Define
\begin{equation*}
\tilde{g}(x) =
\begin{cases}
f(x_i) \quad & \text{if } x\in B_\infty(x_i, \gamma r_i), \\
f(x) \quad & \text{if } x\in {[0,1]}^d \setminus \bigcup_{i=1}^k U_\infty (x_i, r_i).
\end{cases}
\end{equation*}
Hence $\tilde{g}$ is clearly continuous on the closed subspace
\begin{equation*}
\bigcup_{i=1}^k B_\infty(x_i, \gamma r_i) \cup \left( {[0,1]}^d \setminus \bigcup_{i=1}^k U_\infty (x_i, r_i)\right),
\end{equation*}
and we can use Tietze's Theorem to continuously extend $\tilde{g}$ to the whole ${[0,1]}^d$. We denote the extension by $\tilde{g}$ as well.

However, the statement of Tietze's Theorem gives us no control on the distance between $\tilde{g}$ and $f$, and so we need to perform a simple truncation procedure to ensure $g$ will indeed be close to $f$. We define $g\colon {[0,1]}^d \to \field{R}$ as
\begin{equation*}
g(x)=
\begin{cases}
\min \left\{ \max \left\{ \tilde{g}(x), f(x_i)-\frac{\varepsilon}{2} \right\}, f(x_i)+\frac{\varepsilon}{2} \right\} \quad &\text{if } x\in B_\infty(x_i,r_i), \\
\tilde{g}(x) \quad &\text{otherwise}.
\end{cases}
\end{equation*}
To see that $g$ is continuous, take any $i\in\{1,\dots, k\}$ and observe that for each \mbox{$x\in \partial B_\infty(x_i,r_i)$}, we have $\tilde{g}(x)=f(x) \in (f(x_i)-\varepsilon/2, f(x_i)+ \varepsilon/2)$. Therefore, the truncation can only change the function $f$ in the interior of the cubes, so the continuity is preserved.

To see that $\|f-g\|_\infty < \varepsilon$, take any $x \in {[0,1]}^d$. As $g$ coincides with $f$ outside $\bigcup_{i=1}^k B_\infty(x_i, r_i)$, we can assume that $x\in B_\infty (x_i,r_i)$ for some $i$. But then $f(x)\in ( f(x_i)- \varepsilon/2, f(x_i)+ \varepsilon/2)$, and $g(x) \in [ f(x_i)-\varepsilon/2, f(x_i)+\varepsilon/2]$, whence $\abs{f(x)-g(x)} <\varepsilon$. This shows that $A_n$ is dense, and the proof of claim~\cref{Intersection is residual} is complete.

To prove claim~\cref{lip zero almost everywhere}, take any function $f\in \bigcap_{n=1}^\infty A_n$. Our goal is to prove that for almost every $x\in {[0,1]}^d$ we have $\lip f(x)=0$. Since $f\in A_n$ for every $n\in \field{N}$, for each $n$  we can choose
sequences of cubes $\{C_i^{(n)}\}_{i=1}^{k_n}$   in $\mathcal{C}_n$ and of points $\{a_i^{(n)}\}_{i=1}^{k_n}$ such that for all $i\in \{1,\dots, k_n\}$ and all $x\in C_i^{(n)}$ we have
\begin{equation}\label{E:equation1}
\abs{f(x) - a_i^{(n)}} < \beta_n l \left(C_i^{(n)}\right).
\end{equation}
For a positive number $\alpha$ and a cube $C$, we denote by $\alpha C$ the cube that has the same centre as $C$ and satisfies $l(\alpha C)= \alpha \cdot l(C)$. Recalling the definition in \cref{alpha and beta}, we set
\begin{equation*}
Z= \bigcup_{m=1}^\infty \bigcap_{n=m}^\infty \bigcup_{i=1}^{k_n} \alpha_n  C_i^{(n)}.
\end{equation*}
First we observe that $\abs{Z}=1$,  that is the set~$Z$ has full measure in ${[0,1]}^d$. Indeed, for a fixed $n$ we have $\abs{\alpha_n C_i^{(n)}} = \alpha_n^d \abs{C_i^{(n)}}$. As $\alpha_n<1$, the cubes $\alpha_n C_i^{(n)}$ are pairwise disjoint. Therefore,
\begin{align*}
\abs{\bigcup_{i=1}^{k_n} \alpha_n C_i^{(n)}} & = \alpha_n^d \abs{\bigcup_{i=1}^{k_n} C_i^{(n)}} > \alpha_n^d \left( 1 - 2^{-n}\right) \\
& = {\left( 1 - 2^{-n}\right)}^2 > 1-2\cdot 2^{-n}.
\end{align*}
Hence, for a fixed $m\in\field{N}$ we obtain that
\begin{equation*}
\abs{ \bigcap_{n=m}^\infty \bigcup_{i=1}^{k_n} \alpha_n  C_i^{(n)}} > 1-\sum_{n=m}^\infty 2\cdot 2^{-n} = 1- 2^{2-m},
\end{equation*}
and the right-hand side tends to $1$ as $m\to \infty$, which implies the desired conclusion $\abs{Z}=1$.

Finally, we take an arbitrary $x\in Z$; we want to prove that $\lip f(x)=0$. Since $x\in Z$ means that there exists an $m\in\field{N}$ such that for each $n\geq m$ there is an index $i_n\in \{1,\dots, k_n\}$ such that $x\in \alpha_n  C_{i_n}^{(n)}\subset C_{i_n}^{(n)}$.  For any such $n\geq m$, we have from \cref{E:equation1} that for each $y\in C_{i_n}^{(n)}$,
\begin{equation*}
\abs{f(y)-a_{i_n}^{(n)}} < \beta_n l\left(C_{i_n}^{(n)}\right).
\end{equation*}
In particular, since $x\in \alpha_n  C_{i_n}^{(n)}$, we have this for each $y\in B_\infty(x,r_n)$ where
\begin{equation*}
r_n=(1-\alpha_n)\cdot l\left(C_{i_n}^{(n)}\right),
\end{equation*}
and it follows that
\begin{equation*}
\sup_{y\in B_\infty(x,r_n)} \frac{\abs{f(y)-f(x)}}{r_n} \leq \frac{2\beta_n l\left(C_{i_n}^{(n)}\right) }{r_n} = \frac{2\beta_n}{1-\alpha_n}=\frac{2}{n}.
\end{equation*}
As
\begin{equation*}
r_n<l\left(C_{i_n}^{(n)}\right)<\frac{1}{n},
\end{equation*}
$r_n$ tends to~$0$.
It follows that
\begin{equation*}
\liminf_{n\to \infty} \sup_{y\in B_\infty (x,r_n)} \frac{\abs{f(y)-f(x)}}{r_n}=0,
\end{equation*}
which concludes the proof.
\end{proof}

\begin{remark}
In \cite{Zoltan}, the first author studied micro tangent sets of functions defined on the interval~$[0,1]$. Theorem~5 in his paper leads to an alternative
proof of \cref{typical function}
for $d=1$.
\end{remark}

\begin{remark}
One might wonder if it is even true that the typical function has a vanishing $\lip$ everywhere.
However, Lemma~1.1 in \cite{BaloghCsornyei} implies that if $\lip$ of a function vanishes outside a countable set, then the function is differentiable almost everywhere. Thus, for the typical function $f$, there is an uncountable set where $\lip f(x)$ differs from~$0$.
\end{remark}

\subsection*{Acknowledgements}
The first listed author was supported by the Hungarian National Research, Development and Innovation Office--NKFIH, Grant 124749.
Moreover, the research leading to these results has received funding from the European Research Council under the European Union's
Seventh Framework Programme (FP/2007-2013) / ERC Grant
Agreement n.~291497 while the third and fourth authors were research fellows at the University of Warwick.
The finishing process was supported
by the University of Silesia Mathematics Department (Iterative
Functional Equations and Real Analysis program).

\bibliographystyle{siam}
\bibliography{lip}

\end{document}